\documentclass[11pt,a4paper]{article}
\usepackage{mathrsfs}
\usepackage{amsfonts}
\usepackage{}
\usepackage{multirow}
\usepackage{epsfig}
\usepackage{epstopdf}
\usepackage[T1]{fontenc}
\usepackage{geometry}
\usepackage{amsbsy,amsmath,latexsym,amsfonts, epsfig, color, authblk, amssymb, graphics, bm}
\usepackage{epsf,slidesec,epic,eepic}
\usepackage{fancybox}
\usepackage{fancyhdr}
\usepackage{setspace}
\usepackage{cases}
\usepackage{subfigure}
 \usepackage{epsfig}
 \usepackage{epstopdf}
\usepackage{nccmath}
\usepackage{algorithm}
\usepackage{algorithmic}
\usepackage[colorlinks, citecolor=blue]{hyperref}

\newtheorem{theorem}{Theorem}[section]
\newtheorem{lemma}{Lemma}[section]
\newtheorem{definition}{Definition}[section]
\newtheorem{example}{Example}[section]
\newtheorem{proposition}{Proposition}[section]
\newtheorem{corollary}{Corollary}[section]
\newtheorem{assumption}{Assumption}[section]
\newtheorem{remark}{Remark}[section]

\newenvironment{proof}{{\noindent \bf Proof:}}{\hfill$\Box$\medskip}

\definecolor{lred}{rgb}{1,0.8,0.8}
\definecolor{lblue}{rgb}{0.8,0.8,1}
\definecolor{dred}{rgb}{0.6,0,0}
\definecolor{dblue}{rgb}{0,0,0.5}
\definecolor{dgreen}{rgb}{0,0.5,0.5}

 \title{KL property of exponent $1/2$ of $\ell_{2,0}$-norm and DC regularized
 factorizations for low-rank matrix recovery\footnote{This research is supported by the National Natural Science Foundation of China under project
 No.11571120 and No.11701186, and the Natural Science Foundation of Guangdong Province under project
 No.2017A030310418.}}
 \author{Shujun Bi\footnote{(bishj@scut.edu.cn)School of Mathematics, South China University of Technology, China.},
 \ Ting Tao\footnote{School of Mathematics, South China University of Technology, China.}\ \ {\rm and}\ \
 Shaohua Pan\footnote{(shhpan@scut.edu.cn) School of Mathematics, South China University of Technology,
 China.}}

\begin{document}

 \maketitle

 \medskip

 \begin{abstract}
  This paper is concerned with the factorization form of the rank regularized
  loss minimization problem. To cater for the scenario in which only a coarse
  estimation is available for the rank of the true matrix, an $\ell_{2,0}$-norm
  regularized term is added to the factored loss function to reduce the rank adaptively;
  and account for the ambiguities in the factorization, a balanced term is then introduced.
  For the least squares loss, under a restricted condition number assumption
  on the sampling operator, we establish the KL property of exponent $1/2$ of
  the nonsmooth factored composite function and its equivalent DC reformulations
  in the set of their global minimizers. We also confirm the theoretical
  findings by applying a proximal linearized alternating minimization method
  to the regularized factorizations.
 \end{abstract}


 \section{Introduction}\label{sec1}

 Let $\mathbb{R}^{m\times n}$ be the vector space of all $m\times n$
 real matrices, endowed with the trace inner product $\langle\cdot,\cdot\rangle$
 and its induced Frobenius norm $\|\cdot\|_F$. Low-rank matrix recovery problems
 aim to recover a matrix $M\in\mathbb{R}^{m\times n}$ of rank $r$ with
 $1\le r\ll\min(m,n)$ via some noiseless or noisy observations.
 Let $f\!:\mathbb{R}^{m\times n}\rightarrow\mathbb{R}_+$ denote an appropriate
 smooth empirical loss function. When a tight upper estimation,
 say the positive integer $\kappa$, for the rank of $M$ is available,
 it is natural to model low-rank matrix recovery problems as
 \begin{equation}\label{rank-constr}
  \min_{X\in\mathbb{R}^{m\times n}}\Big\{f(X)\ \ {\rm s.t.}\ \ {\rm rank}(X)\le\kappa\Big\};
 \end{equation}
 if only a coarse estimation, say $\min(m,n)$, is available, it is reasonable to model it as
 \begin{equation}\label{rank-reg}
  \min_{X\in\mathbb{R}^{m\times n}}\Big\{\nu f(X)+{\rm rank}(X)\Big\},
 \end{equation}
 where $\nu>0$ is the regularization parameter. Without loss of generality, we stipulate
 that $m\le n$ and assume that the problems \eqref{rank-constr} and \eqref{rank-reg}
 have a nonempty global optimal solution set, denoted by $\mathcal{X}^*$.
 When the loss function $f$ is specified as
 \begin{equation}\label{least-squares}
  f(X):=\frac{1}{2}\|\mathcal{A}(X)-b\|^2
 \end{equation}
 for a sampling operator $\mathcal{A}\!:\mathbb{R}^{m\times n}\rightarrow\mathbb{R}^p$
 and an observation vector $b\in\mathbb{R}^p$, the above \eqref{rank-constr} and \eqref{rank-reg}
 become the popular rank constrained and rank regularized least squares problem,
 respectively. Rank optimization problems have a host of applications in a variety of
 fields such as system identification and control \cite{Fazel02,FPST13},
 signal and image processing \cite{Haeffele14,ChenChi14}, machine learning
 \cite{RSebro-05,Bach08,Koren092}, statistics \cite{Negahban11,KolLT11},
 finance \cite{Pietersz04}, and quantum tomography \cite{Gross11}, just to name a few
(see also the survey \cite{Davenport16} for some further applications).

 \medskip

 Due to the involved combinatorial property, rank optimization problems
 are generally NP-hard \cite{Fazellog} and it is almost impossible to
 seek their global optima with an algorithm of polynomial-time complexity.
 Over the past decade, many convex relaxation approaches have been developed.
 A popular one is the nuclear norm relaxation approach \cite{Candes09,Recht10}
 that yields a low-rank solution via a single tractable nuclear norm
 optimization problem. Observe that the nuclear norm has a relatively
 weak ability even fails to promoting a low-rank solution in some scenarios
 \cite{Ruslan10,MiaoPS16}. Some researchers have proposed convex relaxation
 approaches by solving a sequential convex program arising from a certain
 nonconvex surrogate of rank minimization problems
 (see, e.g., \cite{Fazellog, LaiMJ13, Mohan12}). Recently, Bi and Pan \cite{BiPan17}
 proposed a multi-stage convex relaxation approach to \eqref{rank-reg} by
 the global exact penalty for its MPEC reformulation, which shows that
 several convex relaxation problems are enough to yield a high-quality
 low-rank solution. Although convex relaxation methods have received well study
 from many fields such as information, computer science, statistics, optimization,
 and so on (see, e.g., \cite{Recht10,Candes11,Negahban11,KolLT11,Cai10,Toh10}),
 the computation of tractable convex optimization problems involved in them are
 very expensive since each iterate requires performing a singular value decomposition
 (SVD), which making them computationally prohibitive in large-scale settings \cite{Hsieh14}.

 \medskip

 To overcome the computational bottleneck of the convex relaxation approach,
 inspired by the work \cite{Burer03}, a common way is to reparameterize the $m\times n$
 matrix variable $X$ as $UV^{\mathbb{T}}$ with $U\!\in\mathbb{R}^{m\times\kappa}$
 and $V\!\in\mathbb{R}^{n\times\kappa}$, and achieve a desirable low-rank solution
 by solving a bi-factored nonconvex problem with less variables which is often regularized
 by a term to account for the ambiguities in the factorization $X=UV^{\mathbb{T}}$.
 Corresponding to the rank constrained problem \eqref{rank-constr},
 one usually solves the following smooth nonconvex optimization
 \begin{equation}\label{balance-rconstr}
   \min_{U\in\mathbb{R}^{m\times \kappa}, V\in\mathbb{R}^{n\times \kappa}}
   \Big\{f(UV^{\mathbb{T}})+\frac{\mu}{4}\|U^{\mathbb{T}}U-V^{\mathbb{T}}V\|_F^2\Big\}
 \end{equation}
 where $\mu>0$ is the regularization parameter. Such a reparametrization
 automatically enforces the low-rank structure and leads to low computational
 cost per iteration. Owing to this, the nonconvex factored approach is widely
 used in large-scale applications such as recommendation systems or
 collaborative filtering \cite{Koren091,Koren092}. Though the bi-factored
 problem is nonconvex, it is amenable to local search heuristics such as
 gradient descent or alternating minimization which can exploit its bilinear
 structure well. In fact, numerical experiments indicate that they tend to
 work well in practice \cite{Ge16,Lee16,SunLuo16}.

 \medskip

 Despite the superior empirical performance of the nonconvex factored approach,
 the understanding of its theoretical guarantees is rather limited in comparison
 with the convex relaxation approach since its analysis is well known to be
 notoriously difficult. In the recent years, some active progress has been made
 for the theory of the nonconvex factored approach. One research line focuses on
 the (regularized) factorizations of rank optimization problems from a local view
 and aims to characterize the growth behavior of objective functions around the set
 of global optimal solutions
 (see, e.g., \cite{Jain13,Park16,SunLuo16,Tu16,Zhang18,Zhao15,Zheng16}). A direct consequence
 of this line is that standard methods such as gradient descent and alternating minimization,
 when initialized with a point in the stated neighborhood, will converge linearly
 to an optimal solution. Another line takes a global view and aims to establish
 the geometric landscape of the factorization models of rank optimization problems,
 especially strict saddle point property (see, e.g., \cite{Park17,Ge16,Ge17,LiG18,Li18,ZhangX18,Zhu181}).
 As a result, suitably modified gradient descent methods can be shown to converge
 globally to an optimal solution, and their convergence rates can be established;
 see Jin et al. \cite{Jin17}.

 \medskip

 We notice that most of the above theoretical results focus on
 the exact-parametrization case $\kappa=r$ of the (regularized) factorizations,
 and consequently are applicable to the rank optimization problems
 \eqref{rank-constr} and \eqref{rank-reg} for which the rank of the optimal
 solutions is known. In fact, the most crucial and difficult part for
 most of rank optimization problems is to achieve the rank of global optima
 or its best estimation. This means that the factored model of the rank
 regularized problem \eqref{rank-reg} with a coarse estimation $\kappa$
 is more practical. Motivated by the efficiency of the nuclear norm relaxation approach,
 there are some works \cite{Rennie05,Cabral13,Li18,Zhang18} centering
 around the factorized nuclear norm surrogate problem
 \begin{equation}\label{prob-Fro}
  \min_{U\in\mathbb{R}^{m\times \kappa}, V\in\mathbb{R}^{n\times \kappa}}
  \Big\{\nu f(UV^{\mathbb{T}})+\frac{1}{2}(\|U\|_F^2+\|V\|_F^2)\Big\}
 \end{equation}
 where the positive integer $\kappa$ is an upper estimation for $r$.
 Among others, Li et al. \cite{Li18} took a global view and proved that
 each critical point of the smooth factored problem \eqref{prob-Fro} either
 corresponds to the global optimum of the nuclear norm surrogate problem or
 is a strict saddle point; while Zhang et al. \cite{Zhang18} took a local view
 and established the KL property of exponent $1/2$ of the objective function
 over the set of global optima but only for the exact-parametrization case
 $\kappa=r$ and the loss $f$ in \eqref{least-squares} with a full sampling operator
 $\mathcal{A}$. As previously mentioned, the nuclear norm has a weak ability
 even fails to promoting a low-rank solution in some cases, which implies that
 the Fro-norm regularized factorization model will carry on this potential insufficiencey.

 \medskip

 Motivated by this, in this work we concentrate on the following factored form of \eqref{rank-reg}:
 \begin{equation}\label{prob-balance}
  \min_{U\in\mathbb{R}^{m\times \kappa},V\in\mathbb{R}^{n\times \kappa}}
  \!\Big\{\Psi(U,V):=\nu f(UV^{\mathbb{T}})+\frac{\mu}{4}\|U^{\mathbb{T}}U\!-\!V^{\mathbb{T}}V\|_F^2
  +\!\frac{1}{2}\big(\|U\|_{2,0}+\|V\|_{2,0}\big)\Big\},
 \end{equation}
 where the positive integer $\kappa$ is an upper estimation for $r$,
 and $\|U\|_{2,0}$ means the $\ell_{2,0}$-norm of $U$, i.e., the number
 of nonzero columns of $U$. As will be shown in Section \ref{sec3.1},
 when the set $\mathcal{X}^*$ contains a matrix with rank no more than $\kappa$,
 the factored problem \eqref{prob-balance} is equivalent to \eqref{rank-reg}
 in the sense that their global optima can be obtained from each other. The balanced term
 $\frac{\mu}{4}\|U^{\mathbb{T}}U\!-\!V^{\mathbb{T}}V\|_F^2$ is introduced
 to remove the ambiguities caused by the bilinear form $UV^{\mathbb{T}}$.
 Consider that the discontinuity of $\ell_{2,0}$-norm may cause some inconvenience
 in designing effective algorithms for solving \eqref{prob-balance}. We also reformulate
 \eqref{prob-balance} as a mathematical program with equilibrium constraint (MPEC)
 by the variational characterization of $\ell_{2,0}$-norm, and derive some equivalent
 DC (difference of convexity) reformulations from the global exact penalty of
 the MPEC under a suitable restriction on the loss $f$. The main contribution
 of this work is, for the least squares loss with $b=\mathcal{A}(M)$, to establish
 the KL property of exponent $1/2$ of the function $\Psi$ and its equivalent
 DC surrogates over the set of their global minimizers under a restricted
 condition number assumption on $\mathcal{A}$.

 \medskip

 For nonconvex and nonsmooth functions, the KL property of exponent $1/2$ over
 the set of critical points is a weak growth behavior to guarantee that many
 first-order algorithms for minimizing them converge to a critical point with a linear rate
 (see, e.g., \cite{Attouch10,Attouch13,Bolte14}). From \cite[Section 4]{Attouch10}
 many classes of functions indeed satisfy the KL property, but generally it is
 not an easy task to verify whether these functions have the KL property of exponent
 $1/2$ or not over the set of critical points. Though some useful rules are provided
 in \cite{LiPong18} to identify the exponent of KL functions, in most cases one still
 needs to analyze them case by case. Though we establish the KL property of
 exponent $1/2$ of the function $\Psi$ and its equivalent DC surrogates only in a subset
 of their critical point set, to the best of our knowledge, this work is the first
 to achieve such a growth behavior for the nonsmooth factored function with
 an over-parametrization case $\kappa\ge r$.

 \section{Notation and preliminaries}\label{Sec2}

 Throughout this paper, $\mathbb{O}^{n}$ represents the set of all
 $n\times n$ orthonormal matrices, $e$ denotes a vector of all ones
 whose dimension is known from the context, and ${\rm Diag}(z)$
 for a vector $z$ means a rectangular diagonal matrix.
 For a given $X\in\mathbb{R}^{m\times n}$, $\sigma(X)\in\mathbb{R}^m$
 denotes the singular value vector of $X$ arranged in a nonincreasing order,
 i.e., $\sigma_1(X)\ge\cdots\ge\sigma_m(X)$,
 $\mathbb{O}^{m,n}(X):=\!\{(P,Q)\in \mathbb{O}^m\times\mathbb{O}^n\ |\ X=P{\rm Diag}(\sigma(X))Q^{\mathbb{T}}\}$,
 $\|X\|$ means the spectral norm of $X$, $X_{j}$ denotes the $j$-th column of $X$,
 and for an index set $J\subseteq\{1,\ldots,n\}$,
 $X_{J}$ represents the matrix consisting of those $X_{j}$ with $j\in J$.
 For a given $\delta>0$, $\mathbb{B}(X,\delta)$ denotes the closed ball on
 the Frobenius norm centered at $X$ of radius $\delta$.
 For a linear operator $\mathcal{B}\!:\mathbb{R}^{m\times n}\rightarrow \mathbb{R}^{p}$,
 $\|\mathcal{B}\|$ means the spectral norm of $\mathcal{B}$.
 For a given integer $k\ge 1$, write
 $\Omega_k:=\big\{X\in\mathbb{R}^{m\times n}\ |\ {\rm rank}(X)\le k\big\}$.
 In the sequel, we denote by $\mathscr{L}$ the family
 of proper lower semicontinuous (lsc) convex functions $\phi\!:\mathbb{R}\to(-\infty,+\infty]$ satisfying
 \begin{equation}\label{phi}
   [0,1]\subseteq {\rm int}({\rm dom}\phi),\ \phi(1)=1,\
   t_{\phi}^*=\mathop{\arg\min}_{t\in[0,1]}\phi(t)\ \ {\rm with}\ \phi(t_{\phi}^*)=0.
  \end{equation}
  For each $\phi\in\!\mathscr{L}$, let $\psi\!:\mathbb{R}\to(-\infty,+\infty]$ be
  the associated proper lsc convex function:
 \begin{equation}\label{psi}
  \psi(t):=\!\left\{\!\begin{array}{cl}
                  \phi(t) &\textrm{if}\ t\in [0,1];\\
                   +\infty & \textrm{otherwise},
                 \end{array}\right.
 \end{equation}
 and denote by $\psi^*$ the conjugate function of $\psi$, i.e.,
 $\psi^*(s)\!:=\sup_{t\in\mathbb{R}}\{st-\psi(t)\}$ for $s\in\mathbb{R}$.
  \subsection{Generalized subdifferentials}\label{sec2.1}

  Next we recall from \cite[Definition 8.3]{RW98} the concepts of
  the regular, limiting and horizon subdifferentials of an extended
  real-valued function at a finite-valued point. For an extended real-valued
  $g\!:\mathbb{R}^n\to(-\infty,+\infty]$, write ${\rm dom}\,g:=\{x\in\mathbb{R}^n\ |\ g(x)<\infty\}$.
  \begin{definition}\label{Gsubdiff-def}
   Consider a function $g\!:\mathbb{R}^n\to[-\infty,+\infty]$ and a point
  $x$ with $g(x)$ finite. The regular subdifferential of $g$ at $x$,
  denoted by $\widehat{\partial}g(x)$, is defined as
  \[
    \widehat{\partial}g(x):=\bigg\{v\in\mathbb{R}^n\ \big|\
    \liminf_{x'\to x\atop x'\ne x}
    \frac{g(x')-g(x)-\langle v,x'-x\rangle}{\|x'-x\|}\ge 0\bigg\};
  \]
  the (limiting) subdifferential of $g$ at $x$, denoted by $\partial g(x)$, is defined as
  \[
    \partial g(x):=\Big\{v\in\mathbb{R}^n\ |\  \exists\,x^k\xrightarrow[g]{}x\
    {\rm and}\ v^k\in\widehat{\partial}g(x^k)\ {\rm with}\ v^k\to v\Big\},
  \]
  and the horizon subdifferential of $g$ at $x$, denoted by $\partial^{\infty}g(x)$, is defined as
  \[
    \partial^{\infty}g(x):=\Big\{v\in\mathbb{R}^n\ |\   \exists\,x^k\xrightarrow[g]{}x\
    {\rm and}\  v^k\in\widehat{\partial}g(x^k)\ {\rm with}\ \lambda_kv^k\to v\ {\rm for\ some}\ \lambda_k\downarrow 0\Big\}.
  \]
 \end{definition}
 \begin{remark}\label{remark-Gsubdiff}
  {\bf(i)} The regular and limit subdifferentials are all closed and satisfy
  $\widehat{\partial}g(x)\subseteq\partial g(x)$,
  and the set $\widehat{\partial}g(x)$ is convex but $\partial g(x)$ is
  generally nonconvex. When $g$ is convex, $\widehat{\partial}g(x)=\partial g(x)$
  and is precisely the subdifferential of $g$ at $x$ in the sense of
  convex analysis. When $f$ is nonconvex, there may be
  a big difference between them. For example,
  for the function $g(z)=-\|z\|$ for $z\in\mathbb{R}^n$,
  we have $\widehat{\partial}g(0)=\emptyset$
  but $\partial g(0)=\{v\in\mathbb{R}^n\ |\ \|v\|=1\}$.

  \medskip
  \noindent
  {\bf(ii)} The point $\overline{x}$ at which $0\in\partial g(\overline{x})$
  (respectively, $0\in\widehat{\partial}g(\overline{x})$) is called a limiting
  (respectively, regular) critical point of $g$. By \cite[Theorem 10.1]{RW98},
  a local minimizer of $g$ is necessarily is a regular critical point,
  and consequently a limit critical point.

  \medskip
  \noindent
  {\bf(iii)} By \cite[Corollary 8.11]{RW98}, $g$ is (subdifferentially) regular
  at $\overline{x}$ if and only if $g$ is locally lsc at $\overline{x}$ with
  $\partial g(\overline{x})=\widehat{\partial}g(\overline{x})$ and
  $\partial^{\infty}g(\overline{x})=[\widehat{\partial}g(\overline{x})]^{\infty}$.
 \end{remark}

  The following lemmas focus on the subdifferential of two composite functions.
 \begin{lemma}\label{subdiff-L2norm1}
  Let $h(z):={\rm sign}(\|z\|)$ for $z\in\mathbb{R}^n$.
  Fix an arbitrary $\overline{z}\in\mathbb{R}^n$. Then, we have
  \[
    \widehat{\partial}h(\overline{z})
    =\partial h(\overline{z})
    =\left\{\begin{array}{cl}
       \{0\}^n & {\rm if}\ \overline{z}\ne 0;\\
       \mathbb{R}^n & {\rm if}\ \overline{z}=0.
       \end{array}\right.
  \]
 \end{lemma}
 \begin{proof}
  When $\overline{z}\ne 0$, since $\widehat{\partial}{\rm sign}(t)|_{t=\|\overline{z}\|}
  =\partial{\rm sign}(t)|_{t=\|\overline{z}\|}=\{0\}$,
  by \cite[Exercise 10.7]{RW98} it follows that $\widehat{\partial}h(\overline{z})
  =\partial h(\overline{z})=\{0\}^n$. When $\overline{z}=0$, by Definition \ref{Gsubdiff-def}
  it is easy to calculate that $\widehat{\partial} h(\overline{z})=\mathbb{R}^n$,
  which means that $\partial h(\overline{z})=\mathbb{R}^n$.
  So, the result holds.
 \end{proof}
 \begin{lemma}\label{subdiff-L2norm2}
  Let $h(z)=\vartheta(\|z\|)$ for $z\in\mathbb{R}^n$
  where $\vartheta\!:\mathbb{R}\to\mathbb{R}$ is a locally Lipschitz function.
  Consider an arbitrary point $\overline{z}$ with $\vartheta$ finite at $\|\overline{z}\|$.
  Then, with $g(z)\equiv\|z\|$,
  \[
   \partial h(\overline{z})=\partial\vartheta(\|\overline{z}\|)\frac{\overline{z}}{\|\overline{z}\|}
   \ \ {\rm if}\ \ \overline{z}\ne 0
   \ \ {\rm and}\ \
   \partial h(\overline{z})\subseteq D^*g(\overline{z})\partial\vartheta(\|\overline{z}\|)\ \
   {\rm if}\ \ \overline{z}= 0
  \]
  where $u\in D^*g(\overline{z})(\tau)$ if and only if
  $(u,\tau)\in\mathcal{K}:=\{(\xi,\omega)\in\mathbb{R}^n\times\mathbb{R}\ |\ \omega\ge \|\xi\|\}$.
 \end{lemma}
 \begin{proof}
  Since $\vartheta$ is a locally Lipschitz function, $\partial^{\infty}g(t)=\{0\}$
  for any $t\in\mathbb{R}$. When $\overline{z}\ne 0$, the result follows from \cite[Exercise 10.7]{RW98}.
  When $\overline{z}=0$, by \cite[Theorem 10.49]{RW98} we have
  $\partial h(\overline{z})\subseteq D^*g(\overline{z})\partial\vartheta(\|\overline{z}\|)$,
  where $D^*g(\overline{z})$ is the coderivative of $g$ at $\overline{z}$, defined by
  \begin{equation}\label{Dstarf}
    u\in D^*g(\overline{z})(\tau)\Longleftrightarrow
    (u,-\tau)\in\mathcal{N}_{{\rm gph}g}(\overline{z},\|\overline{z}\|).
  \end{equation}
  We next characterize the limiting normal cone
  $\mathcal{N}_{{\rm gph}g}(\overline{z},\|\overline{z}\|)$.
  Notice that
  \[
    \mathcal{N}_{{\rm gph}g}(\overline{z},\|\overline{z}\|)
    \subseteq \mathcal{N}_{{\rm epi}g}(\overline{z},\|\overline{z}\|)
    =\big[\mathcal{T}_{{\rm epi}g}(\overline{z},\|\overline{z}\|)\big]^{\circ}
    =\mathcal{K}^{\circ},
  \]
  where $[\mathcal{T}_{{\rm epi}g}(\overline{z},\|\overline{z}\|)]^{\circ}$
  is the negative polar cone of $\mathcal{T}_{{\rm epi}g}(\overline{z},\|\overline{z}\|)$,
  the inclusion is due to ${\rm gph}g\subseteq{\rm epi}g$,
  the first equality is due to the convexity of ${\rm epi}g$,
  and the last one is since $\mathcal{T}_{{\rm epi}g}(\overline{z},\|\overline{z}\|)
  ={\rm epi}g'(\overline{z},\cdot)=\mathcal{K}$.
  Now take an arbitrary $(\xi,\omega)\in\mathcal{K}^{\circ}$,
  i.e., $\omega\le-\|\xi\|$. Then, for any ${\rm gph}g\ni(z,t)$ sufficiently
  close to $(\overline{z},\|\overline{z}\|)$,
  $\langle \xi,z\rangle+\omega t =\langle \xi,z\rangle+\omega\|z\|\le 0$. Then,
  \[
    \limsup_{{\rm gph}g\ni(z,t)\to(\overline{z},\|\overline{z}\|)
    \atop (z,t)\ne(\overline{z},\|\overline{z}\|)}\frac{\langle \xi,z\rangle+\omega t}{\sqrt{\|z\|^2+t^2}}\le 0.
  \]
  This shows that $(\xi,\omega)\in\widehat{\mathcal{N}}_{{\rm gph}g}(\overline{z},\|\overline{z}\|)$,
  the Fr\'{e}chet normal cone to ${\rm gph}g$ at $(\overline{z},\|\overline{z}\|)$.
  Since $\widehat{\mathcal{N}}_{{\rm gph}g}(\overline{z},\|\overline{z}\|)
  \subseteq\mathcal{N}_{{\rm gph}g}(\overline{z},\|\overline{z}\|)$,
  it follows that $(\xi,\omega)\in\mathcal{N}_{{\rm gph}g}(\overline{z},\|\overline{z}\|)$.
  By the arbitrariness of $(\xi,\omega)$ in $\mathcal{K}^{\circ}$,
  we have $\mathcal{N}_{{\rm gph}g}(\overline{z},\|\overline{z}\|)=\mathcal{K}^{\circ}$.
  Together with \eqref{Dstarf},  $u\in D^*g(\overline{z})(\tau)$ if and only if
  $(u,\tau)\in\mathcal{K}$. Thus, the desired inclusion follows.
 \end{proof}

 By Remark \ref{remark-Gsubdiff}(iii), $-g$ is not regular since
 $\widehat{\partial}(-g)(0)\ne\partial(-g)(0)$. This means that
 $sg$ for each $s\in\partial\vartheta(0)$ is not regular
 at the origin unless $\partial\vartheta(0)\subseteq\mathbb{R}_{+}$.
 Consequently, the inclusion in Lemma \ref{subdiff-L2norm2} generally
 can not become an equality.
 \subsection{Kurdyka-{\L}ojasiewicz property}\label{sec2.2}

  We recall from \cite{Attouch10} the concept of the KL property of
  an extended real-valued function.
 \begin{definition}\label{KL-Def}
  Let $g\!:\mathbb{R}^n\!\to(-\infty,+\infty]$ be a proper function.
  The function $g$ is said to have the Kurdyka-{\L}ojasiewicz (KL) property
  at $\overline{x}\in{\rm dom}\,\partial g$ if there exist $\eta\in(0,+\infty]$,
  a continuous concave function $\varphi\!:[0,\eta)\to\mathbb{R}_{+}$ satisfying
  the following two conditions
  \begin{itemize}
    \item [(i)] $\varphi(0)=0$ and $\varphi$ is continuously differentiable on $(0,\eta)$;

    \item[(ii)] for all $s\in(0,\eta)$, $\varphi'(s)>0$,
  \end{itemize}
  and a neighborhood $\mathcal{U}$ of $\overline{x}$ such that for all
  \(
    x\in\mathcal{U}\cap\big[g(\overline{x})<g<g(\overline{x})+\eta\big],
  \)
  \[
    \varphi'(g(x)-g(\overline{x})){\rm dist}(0,\partial g(x))\ge 1.
  \]
  If the corresponding $\varphi$ can be chosen as $\varphi(s)=c\sqrt{s}$
  for some $c>0$, then $g$ is said to have the KL property of exponent $1/2$
  at $\overline{x}$. If $g$ has the KL property of exponent $1/2$
  at each point of ${\rm dom}\,\partial g$,
  then $g$ is called a KL function of exponent $1/2$.
 \end{definition}
 \begin{remark}\label{KL-remark}
  By \cite[Lemma 2.1]{Attouch10}, a proper function has the KL property
  of exponent $1/2$ at any noncritical point. Hence, to show that
  it is a KL function of exponent $1/2$, it suffices to check
  whether it has the KL property of exponent $1/2$ at each critical point.
 \end{remark}

 \subsection{Restricted smallest and largest eigenvalues}\label{sec2.3}

  When handling low-rank matrix recovery problems, restricted strong convexity
  and restricted smoothness are common requirement for loss functions
  (see, e.g., \cite{Negahban11,Zhu181,Li18}). For the least squares loss function
  in \eqref{least-squares}, these properties essentially require that the restricted
  smallest and largest eigenvalues of $\mathcal{A}^*\mathcal{A}$ satisfy a certain condition.
  \begin{definition}\label{Def-Rsvalue}
   Let $\mathcal{B}\!:\mathbb{R}^{m\times n}\rightarrow\mathbb{R}^p$ be a given
   linear mapping. The $k$-restricted smallest and largest eigenvalues of
   of $\mathcal{B}^*\mathcal{B}$ are respectively defined as follows:
  \[
    \lambda_{k,\rm min}(\mathcal{B}^*\mathcal{B}):=\min_{X\in\Omega_k,\|X\|_F=1}\|\mathcal{B}(X)\|^2
    \ \ {\rm and}\ \
    \lambda_{k,\rm max}(\mathcal{B}^*\mathcal{B}):=\max_{X\in\Omega_k,\|X\|_F=1}\|\mathcal{B}(X)\|^2,
  \]
  and the ratio $\frac{\lambda_{k,\rm max}(\mathcal{B}^*\mathcal{B})}{\lambda_{k,\rm min}(\mathcal{B}^*\mathcal{B})}$
  is called the $k$-restricted condition number of $\mathcal{B}^*\mathcal{B}$.
  \end{definition}

  Clearly, the least squares loss function \eqref{least-squares}
  has the $k$-restricted strong convexity in \cite{Zhu181,Li18}
  if and only if the $k$-restricted smallest eigenvalue of
  $\mathcal{A}^*\mathcal{A}$ is positive. When the $k$-restricted smallest
  and largest eigenvalues of $\mathcal{A}^*\mathcal{A}$ satisfy
  $\lambda_{k,\rm min}(\mathcal{A}^*\mathcal{A})=1-\delta_{k}$
  and $\lambda_{k,\rm max}(\mathcal{A}^*\mathcal{A})=1+\delta_{k}$ for
  some $\delta_k\in[0,1)$, the sampling operator $\mathcal{A}$ satisfies
  the restricted isometry property (RIP). Thus, we conclude from \cite{Recht10} that
  for many types of random sampling operators, there is a high probability
  for $\mathcal{A}^*\mathcal{A}$ to have a good restricted condition number,
  that is, the value of $\lambda_{r,\rm max}(\mathcal{A}^*\mathcal{A})
  /\lambda_{r,\rm min}(\mathcal{A}^*\mathcal{A})$ is not too large.

  \medskip

  For the least squares loss function \eqref{least-squares}, when $\mathcal{A}^*\mathcal{A}$
  has a positive $r$-restricted smallest eigenvalue, we have the following result
  which improves the result of \cite[Proposition 1]{Li18}.
  \vspace{-0.5cm}
 \begin{lemma}\label{RIPXY}
  For the loss function \eqref{least-squares}, let $\alpha$ and
  $\beta$ be the $r$-restricted smallest and largest eigenvalues of
  $\mathcal{A}^*\mathcal{A}$, respectively. If $\alpha>0$,
  for any $X,Y\in\mathbb{R}^{m\times n}$ with ${\rm rank}([X\ Y])\le r$,
  \[
    \Big|\frac{2}{\alpha+\beta}\langle \mathcal {A}(X),\mathcal {A}(Y)\rangle
     -\langle X,Y\rangle\Big|\leq \frac{\beta-\alpha}{\beta+\alpha}\|X\|_F\|Y\|_F.
  \]
 \end{lemma}
 \begin{proof}
  Fix an arbitrary $X,Y\in\mathbb{R}^{m\times n}$ with ${\rm rank}([X\ Y])\le r$. If one of $X$ and $Y$
  is the zero matrix, the result is trivial. So, we assume that $X\ne 0$
  and $Y\ne 0$. Write $\overline{X}:=\frac{X}{\|X\|_F}$ and $\overline{Y}:=\frac{Y}{\|Y\|_F}$.
  Notice that ${\rm rank}([\overline{X}\ \ \overline{Y}])\le r$
  and ${\rm rank}(\overline{X}\pm\overline{Y})\leq {\rm rank}([\overline{X}\ \ \overline{Y}])$.
  Then, we have
  \begin{subequations}
  \begin{align*}
   \alpha\|\overline{X}+\overline{Y}\|_F^2
   \le\big\|\mathcal {A}\big(\overline{X}+\overline{Y}\big)\big\|^2
   \le\beta\big\|\overline{X}+\overline{Y}\|_F^2,\\
   \alpha\|\overline{X}-\overline{Y}\|_F^2
   \le\big\|\mathcal {A}\big(\overline{X}-\overline{Y}\big)\big\|^2
   \le\beta\big\|\overline{X}+\overline{Y}\|_F^2.
  \end{align*}
  \end{subequations}
  Together with $4|\big\langle\mathcal{A}(\overline{X}),\mathcal{A}(\overline{Y})\big\rangle|
  =\big|\|\mathcal{A}(\overline{X}+\overline{Y})\|^2-\|\mathcal{A}(\overline{X}-\overline{Y})\|^2\big|$,
  if follows that
 \begin{subequations}
  \begin{align*}
  4\big\langle\mathcal{A}(\overline{X}),\mathcal {A}(\overline{Y})\big\rangle
  \le\beta\big\|\overline{X}+\overline{Y}\|_F^2-\alpha\|\overline{X}-\overline{Y}\|_F^2
  =2(\beta-\alpha)+2(\beta+\alpha)\big\langle \overline{X},\overline{Y}\big\rangle,\\
  -4\big\langle\mathcal {A}(\overline{X}),\mathcal {A}(\overline{Y})\big\rangle
  \le\beta\big\|\overline{X}-\overline{Y}\|_F^2-\alpha\big\|\overline{X}+\overline{Y}\|_F^2
  =2(\beta-\alpha)-2(\beta+\alpha)\big\langle \overline{X},\overline{Y}\big\rangle.\nonumber
 \end{align*}
 \end{subequations}
  The last two inequalities imply the desired inequality.
  The proof is then completed.
 \end{proof}

  To close this section, we characterize the rank function
  in terms of the $\ell_{2,0}$-norm.
  \begin{lemma}\label{rank-char}
   Given a matrix $X\in\mathbb{R}^{m\times n}$. If ${\rm rank}(X)\le\kappa$
   for an integer $\kappa\ge 1$, then
   \begin{equation}\label{rank-equa}
    {\rm rank}(X)=\min_{R\in\mathbb{R}^{m\times\kappa},L\in\mathbb{R}^{n\times\kappa}}
    \Big\{\frac{1}{2}\big(\|R\|_{2,0}+\|L\|_{2,0}\big)\!: X=RL^{\mathbb{T}}\Big\}.
   \end{equation}
 \end{lemma}
 \begin{proof}
  Take an arbitrary feasible point
  $(R,L)\in\mathbb{R}^{m\times\kappa}\times\mathbb{R}^{n\times\kappa}$ of \eqref{rank-equa}.
  Notice that $X=RL^{\mathbb{T}}={\textstyle\sum_{j=1}^\kappa}R_{j}L_{j}^{\mathbb{T}}$
  and there are at most $\min(\|R\|_{2,0},\|L\|_{2,0})$ nonzero terms.
  Hence, ${\rm rank}(X)\le\min(\|R\|_{2,0},\|L\|_{2,0})$ and
  ${\rm rank}(X)\le \frac{1}{2}\big(\|R\|_{2,0}+\|L\|_{2,0}\big)$.
  Since $(R,L)$ is an arbitrary feasible point, this shows that
  ${\rm rank}(X)$ is a lower bound for the objective function of \eqref{rank-equa}
  over its feasible set. Also, by taking $(U,V)\in\mathbb{O}^{m,n}(X)$
  and setting $\overline{R}=\big[\sqrt{\sigma_1(X)}U_{1}\ \cdots\ \sqrt{\sigma_\kappa(X)}U_{\kappa}\big]$
  and $\overline{L}=\big[\sqrt{\sigma_1(X)}V_{1}\ \cdots\ \sqrt{\sigma_\kappa(X)}V_{\kappa}\big]$,
  we have
  \[
    \|\overline{R}\|_{2,0}=\|\overline{L}\|_{2,0}={\rm rank}(X)\ \ {\rm and}\ \
    X=\overline{R}\,\overline{L}^{\mathbb{T}}.
  \]
  This shows that the optimal value of \eqref{rank-equa} equals ${\rm rank}(X)$.
  The desired result holds.
 \end{proof}
 \section{Factored reformulations}\label{sec3}

  We provide several factored reformulations of the rank regularized
  problem \eqref{rank-reg} by means of the $\ell_{2,0}$-norm of matrices and
  its variational characterization. Recall the definition of the function family
  $\mathscr{L}$. With an arbitrary $\phi\in\!\mathscr{L}$, it is easy to check that
  for any $z\in\mathbb{R}^{\kappa}$,
 \begin{equation*}
  \|z\|_0=\min_{w\in \mathbb{R}^{\kappa}}\Big\{{\textstyle\sum_{i=1}^{\kappa}}\phi(w_i):
  \ 0\leq w\leq e,\ \langle e-w, |z|\rangle =0\Big\}.
  \end{equation*}
  Consequently, for any $Z\in\mathbb{R}^{m\times n}$,
  with $\mathcal{G}(Z):=(\|Z_{1}\|,\|Z_{2}\|,\ldots,\|Z_{n}\|)^{\mathbb{T}}$
  it holds that
  \begin{equation}\label{group-20}
  \|Z\|_{2,0}=\min_{w\in \mathbb{R}^{n}}\Big\{{\textstyle\sum_{i=1}^{n}}\phi(w_i):
   \ 0\leq w\leq e,\,\langle e-w,\mathcal{G}(Z)\rangle =0\Big\}.
  \end{equation}
  This provides a variational characterization for the $\ell_{2,0}$-norm of
  matrices. Such a characterization was exploited in \cite{BiPan18} to design a convex
  relaxation approach to group sparsity.
 \subsection{$\ell_{2,0}$-norm regularized factorization}\label{sec3.1}

  We first argue that the rank regularized problem \eqref{rank-reg} can be
  reformulated as an equivalent $\ell_{2,0}$-norm regularized factorization model.
  This is implied by the following lemma.
  \begin{lemma}\label{lemma-zreg}
   If $X^*$ is a global optimal solution of rank $r$ for the problem \eqref{rank-reg},
   then $(R^*,L^*)$ with $R^*=[\sqrt{\sigma_1(X^*)}U_{1}^*\ \cdots\,\sqrt{\sigma_\kappa(X^*)}U_{\kappa}^*]$
   and $L^*=[\sqrt{\sigma_1(X^*)}V_{1}^*\ \cdots\ \sqrt{\sigma_\kappa(X^*)}V_{\kappa}^*]$
   for $(U^*,V^*)\!\in\!\mathbb{O}^{m,n}(X^*)$ is globally optimal to
   the following problem with $\kappa\ge r$
   \begin{equation}\label{prob-zreg}
    \min_{U\in\mathbb{R}^{m\times\kappa},V\in\mathbb{R}^{n\times\kappa}}
    \Big\{\nu f(UV^{\mathbb{T}})+\frac{1}{2}\big(\|U\|_{2,0}+\|V\|_{2,0}\big)\Big\}.
  \end{equation}
  Conversely, if $\mathcal{X}^*\cap\Omega_{\kappa}\ne\emptyset$
  and $(\overline{U},\overline{V})$ is a global optimal solution of \eqref{prob-zreg},
  then $\overline{X}=\overline{U}\,\overline{V}^{\mathbb{T}}$ is globally optimal
  to the problem \eqref{rank-reg}.
  \end{lemma}
 \begin{proof}
  Fix an arbitrary $(U,V)\in\mathbb{R}^{m\times\kappa}\times\mathbb{R}^{n\times\kappa}$
  and write $X=UV^{\mathbb{T}}$. By Lemma \ref{rank-char},
  $\|U\|_{2,0}+\|V\|_{2,0}\ge 2{\rm rank}(X)$, which along with
  the global optimality of $X^*$ implies that
  \begin{align*}
   \nu f(UV^{\mathbb{T}})+\frac{1}{2}\big(\|U\|_{2,0}+\|V\|_{2,0}\big)
   \ge \nu f(X)+{\rm rank}(X)&\ge \nu f(X^*)+{\rm rank}(X^*).
  \end{align*}
  Notice that $R^*L^*{^{\mathbb{T}}}=X^*$ and $2{\rm rank}(X^*)=\|R^*\|_{2,0}+\|L^*\|_{2,0}$.
  From the last inequality,
  \[
    \nu f(UV^{\mathbb{T}})+\frac{1}{2}\big(\|U\|_{2,0}+\|V\|_{2,0}\big)
   \ge\nu f(R^*L^*{^{\mathbb{T}}})+\frac{1}{2}\big(\|R^*\|_{2,0}+\|L^*\|_{2,0}\big).
  \]
  By the arbitrariness of $(U,V)$, this shows that $(R^*,L^*)$ is globally optimal to \eqref{prob-zreg}.

  \medskip

  Conversely, let $X^*$ be an arbitrary point from $\mathcal{X}^*\cap\Omega_{\kappa}$
  and let $(U^*,V^*)\in\mathbb{O}^{m,n}(X^*)$.
  Write $R=\big[\sqrt{\sigma_1(X^*)}U_{1}^*\ \cdots\ \sqrt{\sigma_{\kappa}(X^*)}U_{\kappa}^*\big]$
  and $L=\big[\sqrt{\sigma_1(X^*)}V_{1}^*\ \cdots\ \sqrt{\sigma_{\kappa}(X^*)}V_{\kappa}^*\big]$.
  Then, it holds that $\|R\|_{2,0}+\|L\|_{2,0}=2{\rm rank}(X^*)$. Consequently, we have
  \begin{align*}
   \nu f(X^*)+{\rm rank}(X^*)
   &=\nu f(RL^{\mathbb{T}})+\frac{1}{2}\big(\|R\|_{2,0}+\|L\|_{2,0}\big)\\
   &\ge \nu f(\overline{U}\overline{V}^{\mathbb{T}})+\frac{1}{2}\big(\|\overline{U}\|_{2,0}+\|\overline{V}\|_{2,0}\big)\\
   &\ge \nu f(\overline{X})+{\rm rank}(\overline{X})\ \ {\rm with}\ \overline{X}=\overline{U}\overline{V}^{\mathbb{T}},
  \end{align*}
  where the last inequality is by Lemma \ref{rank-char}.
  So, $\overline{U}\overline{V}^{\mathbb{T}}$ is globally optimal to \eqref{rank-reg}.
 \end{proof}

  Lemma \ref{lemma-zreg} shows that if an upper bound $\kappa$ is available
  for a low-rank global optimal solution of \eqref{rank-reg}, seeking such a
  low-rank global optimal solution is equivalent to finding a global optimal
  solution of the $\ell_{2,0}$-norm regularized factorization model \eqref{prob-zreg}.
  Thus, to achieve a low-rank global optimal solution of \eqref{rank-reg} with balanced factors,
  i.e., $\overline{X}=\overline{U}\overline{V}^{\mathbb{T}}$
  with $\overline{U}^{\mathbb{T}}\overline{U}=\overline{V}^{\mathbb{T}}\overline{V}$,
  one may solve the $\ell_{2,0}$-norm regularized factorization model \eqref{prob-balance}.
 The following lemma builds a bridge for the global optimal solution set
 of \eqref{prob-balance} and \eqref{rank-reg}.
 \begin{lemma}\label{lemma-balance}
  Suppose $\mathcal{X}^*\cap\Omega_{\kappa}\ne\emptyset$.
  Then, the optimal solution set of \eqref{prob-balance} has the form
  \[
  \mathcal{W}^*\!:=\left\{(\overline{U},\overline{V})\ |\ \overline{U}\overline{V}^{\mathbb{T}}\in\mathcal{X}^*\cap\Omega_{\kappa},\, \overline{U}^{\mathbb{T}}\overline{U}=\overline{V}^{\mathbb{T}}\overline{V},\, \|\overline{U}\|_{2,0}=\|\overline{V}\|_{2,0}={\rm rank}(\overline{U}\overline{V}^{\mathbb{T}})\right\}.
  \]
 \end{lemma}
 \begin{proof}
  Take an arbitrary $X^*\!\in\mathcal{X}^*\cap\Omega_{\kappa}$. By Lemma \ref{rank-char},
  for any $(U,V)\in\mathbb{R}^{m\times \kappa}\times\mathbb{R}^{n\times \kappa}$,
  \begin{align}
   &\nu f(UV^{\mathbb{T}})+\frac{\mu}{4}\|U^{\mathbb{T}}U-V^{\mathbb{T}}V\|_F^2
    +\frac{1}{2}(\|U\|_{2,0}+\|V\|_{2,0})\nonumber\\
   &\ge f(UV^{\mathbb{T}})+\frac{1}{2}(\|U\|_{2,0}+\|V\|_{2,0})\nonumber\\
   &\geq \nu f(UV^{\mathbb{T}})+{\rm rank}(UV^{\mathbb{T}})\ge \nu f(X^*)+{\rm rank}(X^*).
  \end{align}
  Moreover, when $(U,V)=(\overline{U},\overline{V})$ for an arbitrary
  $(\overline{U},\overline{V})$ from $\mathcal {W}^*$, it holds that
  \begin{align*}
    &\nu f(\overline{U}\overline{V}^{\mathbb{T}})
    +\frac{\mu}{4}\|\overline{U}^{\mathbb{T}}\overline{U}-\overline{V}^{\mathbb{T}}\overline{V}\|_F^2
    +\frac{1}{2}(\|\overline{U}\|_{2,0}+\|\overline{V}\|_{2,0})\\
    &=\nu f(\overline{U}\overline{V}^{\mathbb{T}})+{\rm rank}(\overline{U}\,\overline{V}^{\mathbb{T}})
    =\nu f(X^*)+{\rm rank}(X^*)
  \end{align*}
  where the last equality is due to $\overline{U}\,\overline{V}^{\mathbb{T}}\in\mathcal{X}^*$.
  The last two equations show that the problems \eqref{prob-balance} and \eqref{rank-reg}
  have the same optimal value, and by the arbitrariness of $(\overline{U},\overline{V})$
  in $\mathcal{W}^*$, we conclude that $\mathcal {W}^*$ is included in the set of global
  optimal solutions to \eqref{prob-balance}. Thus, it suffices to argue that the converse
  inclusion holds. For this purpose, let $(\overline{U},\overline{V})$ be an arbitrary
  global optimal solution to \eqref{prob-balance}. Then, by using the fact that
  the problems \eqref{prob-balance} and \eqref{rank-reg} have the same optimal value
  and Lemma \ref{rank-char},
  it follows that
  \begin{align}\label{temp-ineq31}
   \nu f(X^*)+{\rm rank}(X^*)
   &=\nu f(\overline{U}\,\overline{V}^{\mathbb{T}})
    +\frac{\mu}{4}\|\overline{U}^{\mathbb{T}}\overline{U}-\overline{V}^{\mathbb{T}}\overline{V}\|_F^2
    +\frac{1}{2}(\|\overline{U}\|_{2,0}+\|\overline{V}\|_{2,0})\nonumber\\
    &\ge\nu f(\overline{U}\,\overline{V}^{\mathbb{T}})+{\rm rank}(\overline{U}\,\overline{V}^{\mathbb{T}})
    \ge\nu f(X^*)+{\rm rank}(X^*).
  \end{align}
  This implies that $\overline{U}\,\overline{V}^{\mathbb{T}}\in\mathcal{X}^*$
  and the inequalities in \eqref{temp-ineq31} become the equalities. Then,
  \[
   \nu f(\overline{U}\,\overline{V}^{\mathbb{T}})
    +\frac{\mu}{4}\|\overline{U}^{\mathbb{T}}\overline{U}-\overline{V}^{\mathbb{T}}\overline{V}\|_F^2
    +\frac{1}{2}(\|\overline{U}\|_{2,0}+\|\overline{V}\|_{2,0})
  =\nu f(\overline{U}\,\overline{V}^{\mathbb{T}})+{\rm rank}(\overline{U}\,\overline{V}^{\mathbb{T}}).
  \]
  Along with $\frac{1}{2}(\|\overline{U}\|_{2,0}+\|\overline{V}\|_{2,0})\ge{\rm rank}(\overline{U}\,\overline{V}^{\mathbb{T}})$
  by Lemma \ref{rank-char}, we deduce that
  $\overline{U}^{\mathbb{T}}\overline{U}=\!\overline{V}^{\mathbb{T}}\overline{V}$
  and $\frac{1}{2}(\|\overline{U}\|_{2,0}+\|\overline{V}\|_{2,0})
  =\!{\rm rank}(\overline{U}\,\overline{V}^{\mathbb{T}})$, which implies that
  $\|\overline{U}\|_{2,0}=\|\overline{V}\|_{2,0}={\rm rank}(\overline{U}\,\overline{V}^{\mathbb{T}})$.
  That is, $(\overline{U},\overline{V})\in\mathcal{W}^*$. Hence,
  the desired converse inclusion holds.
 \end{proof}
 \begin{remark}\label{remark31}
  Lemma \ref{lemma-zreg} and \ref{lemma-balance} show that
  a global optimal solution of \eqref{rank-reg} can be obtained from
  that of \eqref{prob-zreg} or \eqref{prob-balance} when
  $\mathcal{X}^*\cap\Omega_{\kappa}\ne\emptyset$.
  It is easy to verify that every global optimal solution $(\overline{U},\overline{V})$
  of \eqref{prob-zreg} with $\overline{U}^{\mathbb{T}}\overline{U}=\overline{V}^{\mathbb{T}}\overline{V}$
  is necessarily a global optimal solution of \eqref{prob-balance},
  and when $\mathcal{X}^*\cap\Omega_{\kappa}\ne\emptyset$,
  by Lemma \ref{lemma-balance} and the first part of Lemma \ref{lemma-zreg},
  one may construct a global optimal solution of \eqref{prob-zreg} from every
  global optimal solution of \eqref{prob-balance}.
 \end{remark}

  Lemma \ref{lemma-balance} implies that if the set $\mathcal{X}^*\cap\Omega_{\kappa}$
  can be characterized, one may achieve the global optimal solution set of \eqref{prob-balance}.
  The following proposition states that for the function $f$ specified as in \eqref{least-squares}
  with $b=\mathcal{A}(M)$, the set $\mathcal{X}^*\cap\Omega_k$ can be characterized
  under a suitable condition for the restricted smallest eigenvalue,
  and so is the set $\mathcal{W}^*$.
  \begin{proposition}\label{MW-set}
  Suppose that the function $f$ is given by \eqref{least-squares} with $b=\mathcal{A}(M)$
  for a matrix $M$ of rank $r$, and $\mathcal{X}^*\cap\Omega_{\kappa}\ne\emptyset$.
  If the $2r$-restricted smallest eigenvalue $\alpha$ of the linear operator
  $\mathcal{A}^*\mathcal{A}$ satisfies $\alpha>\frac{2}{\nu\sigma_r^2(M)}$,
  then $\mathcal{X}^*\cap\Omega_k=\{M\}$, and consequently
  \[
    \mathcal{W}^*:=\left\{(\overline{U},\overline{V})\ |\ \overline{U}\,\overline{V}^{\mathbb{T}}=M,\, \overline{U}^{\mathbb{T}}\overline{U}=\overline{V}^{\mathbb{T}}\overline{V},\, \|\overline{U}\|_{2,0}=\|\overline{V}\|_{2,0}=r\right\}.
  \]
 \end{proposition}
 \begin{proof}
  Take an arbitrary $\overline{X}\in \mathcal {X}^*\cap\Omega_{\kappa}$ and write
  $\overline{r}:={\rm rank}(\overline{X})$. By the expression of $f$,
  \begin{equation}\label{temp-ineq32}
    \nu f(\overline{X})+{\rm rank}(\overline{X})\le \nu f(M)+{\rm rank}(M)={\rm rank}(M)=r.
  \end{equation}
  This implies that $\overline{r}\le r$. If $\overline{r}< r$,
  from the fact that $\alpha$ is the $2r$-restricted smallest eigenvalue
  of $\mathcal{A}^*\mathcal{A}$ and $\min_{{\rm rank}(X)\le\overline{r}}\|X-M\|_F^2
  =\sum_{i=\overline{r}+1}^r[\sigma_i(M)]^2$ it follows that
  \begin{align*}
    \nu f(\overline{X})+{\rm rank}(\overline{X})
    &\ge\frac{1}{2}\nu\alpha\big\|\overline{X}\!-M\big\|_F^2+{\rm rank}(\overline{X})\\
    &\ge\frac{1}{2}\nu\alpha(r-\overline{r})[\sigma_r(M)]^2+\overline{r}>r,
  \end{align*}
  where the last inequality is due to $\alpha>\frac{2}{\nu\sigma_r^2(M)}$
  and $r-\overline{r}>0$. This gives a contradiction to the inequality \eqref{temp-ineq32}.
  Consequently, ${\rm rank}(\overline{X})=\overline{r}=r$,
  and $f(\overline{X})= 0$ follows from \eqref{temp-ineq32}.
  Together with $f(\overline{X})\ge\frac{1}{2}\alpha\|\overline{X}-M\|^2$,
  we obtain $\overline{X}=M$. By the arbitrariness
  of $\overline{X}\in\mathcal{X}^*\cap\Omega_{\kappa}$,
  it follows that $\mathcal{X}^*\cap\Omega_{\kappa}=\{M\}$.
  The proof is completed.
 \end{proof}

 \subsection{DC regularized factorizations}\label{sec3.2}

  By Remark \ref{remark31}, one may achieve a global optimal solution of \eqref{rank-reg}
  by solving the $\ell_{2,0}$-norm regularized factorization model \eqref{prob-zreg}
  or its balanced formulation \eqref{prob-balance}. Write
  \begin{equation}\label{FUV}
   F(U,V):=\nu f(UV^{\mathbb{T}}) +\frac{\mu}{4}\|U^{\mathbb{T}}U-V^{\mathbb{T}}V\|_F^2.
  \end{equation}
  Fix an arbitrary $\phi\in\!\mathscr{L}$. By \eqref{group-20},
  the problem \eqref{prob-balance} is equivalent to the following problem
  \begin{align}\label{MPEC-balance}
  &\min_{(U,u)\in\mathbb{R}^{m\times\kappa}\times\mathbb{R}^{\kappa}
  \atop (V,v)\in\mathbb{R}^{n\times\kappa}\times\mathbb{R}^{\kappa}}
   F(U,V)+\frac{1}{2}\sum_{i=1}^{\kappa}\big(\phi(u_i)+\phi(v_i)\big)\nonumber\\
  &\qquad\ {\rm s.t.}\quad\ 0\le u\le e,\,\langle e-u,\mathcal{G}(U)\rangle=0,\nonumber\\
  &\qquad\qquad\quad 0\le v\le e,\,\langle e-v,\mathcal{G}(V)\rangle=0
 \end{align}
 in the sense that if $(\overline{U},\overline{V})$ is a global optimal solution
 to \eqref{prob-balance}, then $(\overline{U},\overline{V}, \overline{u}, \overline{v})$
 with $\overline{u}=\max({\rm sign}(\mathcal {G}(\overline{U})),t^*_{\phi}e)$ and
 $\overline{v}=\max({\rm sign}(\mathcal {G}(\overline{V})),t^*_{\phi}e)$ is globally
 optimal to \eqref{MPEC-balance}; and conversely, if $(\overline{U},\overline{V},\overline{u},\overline{v})$
 is a global optimal solution of \eqref{MPEC-balance}, then $(\overline{U},\overline{V})$
 is globally optimal to \eqref{prob-balance}. Furthermore, the problems \eqref{prob-balance}
 and \eqref{MPEC-balance} have the same optimal value. The problem \eqref{MPEC-balance}
 is an MPEC involving the equilibrium constraints $\langle e-\!u,\mathcal{G}(U)\rangle=0,e-\!u\ge 0$
 and $\langle e-\!v,\mathcal{G}(V)\rangle=0,e-v\ge0$.
 The equivalence between \eqref{prob-balance} and \eqref{MPEC-balance}
 reveals that the combinatorial property of $\|U\|_{2,0}$ and $\|V\|_{2,0}$
 arises from the equilibrium constraints.

 \medskip

 It is well known that handling nonconvex constraints is much harder than handling
 nonconvex objective functions. So, we consider the following penalized problem
 of \eqref{MPEC-balance}
 \begin{align}\label{MPEC-penalty}
  &\min_{(U,u)\in\mathbb{R}^{m\times\kappa}\times\mathbb{R}^{\kappa}
  \atop (V,v)\in\mathbb{R}^{m\times\kappa}\times\mathbb{R}^{\kappa}}\!
  \Big\{F(U,V)+\!\frac{1}{2}\sum_{j=1}^{\kappa}\!\big[\phi(u_j)\!+\phi(v_j)
   +\rho(1\!-\!u_j)\|U_{j}\|+\rho(1\!-\!v_j)\|V_{j}\|\big]\Big\}\nonumber\\
  &\qquad\ {\rm s.t.}\quad\  0\leq u\leq e,\ 0\leq v\leq e.
 \end{align}
 By using \cite[Theorem 3.2]{LiuBP18}, we can establish the following
 global exact penalty result.
 \begin{proposition}\label{prop-Epenalty}
  Let $\phi\in\!\mathscr{L}$. If $\widetilde{f}(U,V)\!:=f(UV^{\mathbb{T}})$
  for $(U,V)\in\mathbb{R}^{m\times\kappa}\times\mathbb{R}^{n\times\kappa}$
  is coercive and $f$ is continuously differentiable in $\mathbb{R}^{m\times n}$,
  then there exists $\widehat{\rho}>0$ such that the problem \eqref{MPEC-penalty}
  associated to each $\rho>\widehat{\rho}$ has the same optimal solution set
  as \eqref{MPEC-balance} does.
 \end{proposition}
 \begin{proof}
  By the coerciveness of $\widetilde{f}$, there exists a constant $\omega>0$
  such that \eqref{MPEC-balance} and \eqref{MPEC-penalty}
  are equivalent to their respective version in which the variables $U$ and $V$
  are restricted to lie in $\mathbb{B}(0,\omega)$. Thus, the conclusion follows
  by \cite[Theorem 3.2]{LiuBP18}.
 \end{proof}

 Recall the definition of $\psi$ in \eqref{psi}. The problem \eqref{MPEC-penalty}
 can be compactly written as
 \[
  \min_{(U,u)\in\mathbb{R}^{m\times\kappa}\times\mathbb{R}^{\kappa}
  \atop (V,v)\in\mathbb{R}^{m\times\kappa}\times\mathbb{R}^{\kappa}}\!
  \Big\{F(U,V)+\!\frac{1}{2}\sum_{j=1}^{\kappa}\!\big[\psi(u_j)\!+\psi(v_j)
   +\rho(1\!-\!u_j)\|U_{j}\|+\rho(1\!-\!v_j)\|V_{j}\|\big]\Big\}
 \]
  which, by introducing the function $\theta(t):=|t|-\psi^*(|t|)$ for $t\in\mathbb{R}$,
  is simplified as
 \begin{equation}\label{prob-DC}
   \min_{U\in\mathbb{R}^{m\times \kappa},V\in\mathbb{R}^{n\times \kappa}}
   \Big\{\Theta_{\rho}(U,V):=F(U,V)+\frac{1}{2}\sum_{j=1}^{\kappa}\big[\theta(\rho\|U_j\|)+\theta(\rho\|V_j\|)\big]\Big\}.
 \end{equation}
  Notice that $\sum_{j=1}^{\kappa}\big[\theta(\rho\|U_j\|)+\theta(\rho\|V_j\|)\big]$
  is a DC function since $U_j\mapsto \psi^*(\rho\|U_j\|)$ is convex
  by the nondecreasing and convexity of $\psi^*$. By Theorem \ref{prop-Epenalty},
  the following result holds.
 \begin{corollary}\label{corollary1-Epenalty}
  Let $\phi\in\!\mathscr{L}$. If the function $\widetilde{f}$ defined in
  Proposition \ref{prop-Epenalty} is coercive and $f$ is continuously
  differentiable in $\mathbb{R}^{m\times n}$,
  then there exists $\widehat{\rho}>0$ such that the problem \eqref{prob-DC}
  associated to each $\rho>\widehat{\rho}$ has the same optimal solution set
  as \eqref{prob-balance} does.
 \end{corollary}


 Combining the above discussions with Remark \ref{remark31}, we conclude that
 when $f$ satisfies the requirement of Proposition \ref{prop-Epenalty} and
 $\mathcal{X}^*\cap\Omega_{\kappa}\ne\emptyset$, one may achieve a global optimal
 solution of \eqref{rank-reg} by solving the penalized problem \eqref{MPEC-penalty}
 or its DC reformulation \eqref{prob-DC}. However, the function $\widetilde{f}$
 associated to many $f$ is not coercive, say, the least squares loss function
 in \eqref{least-squares}. In this case, it is natural to ask what conditions
 can ensure that the problem \eqref{MPEC-penalty} is still a global exact penalty
 of \eqref{MPEC-balance}. The following proposition provides such a condition
 when $f$ is specified as the function in \eqref{least-squares} with $b=\mathcal{A}(M)$.
 \begin{proposition}\label{Epenalty-LQ}
  Suppose that the function $f$ is given by \eqref{least-squares} with $b=\mathcal{A}(M)$
  for a matrix $M$ of rank $r$. If $\mathcal{X}^*\cap\Omega_{\kappa}\ne\emptyset$ and
  the $2r$-restricted smallest eigenvalue $\alpha$ of $\mathcal{A}^*\mathcal{A}$
  satisfies $\alpha>\frac{2}{\nu\sigma_r^2(M)}$,
  then for each $\phi\in\!\mathscr{L}$ the problem \eqref{MPEC-penalty} is
  a global exact penalty of \eqref{MPEC-balance} with threshold $\overline{\rho}:=\max\big(1,\frac{\sqrt{\nu}\|\mathcal {A}\|\sqrt{\kappa}}{\sqrt{\nu\alpha}\sigma_r(M)-\sqrt{2}}\sqrt{1+\!\frac{2\sqrt{r}}{\sqrt{\mu}}}\big)\phi_-'(1)$,
  and consequently the problem \eqref{prob-DC} associated to every $\rho>\overline{\rho}$
  has the same global optimal solution set as \eqref{prob-balance} does.
 \end{proposition}
 \begin{proof}
  We first argue that for all $\rho>\overline{\rho}$ and $(U,V,u,v)\in\mathbb{R}^{m\times\kappa}
  \times\mathbb{R}^{n\times \kappa}\times [0,e]\times[0,e]$,
  \begin{align}\label{UPC}
   F(U,V)+\frac{1}{2}{\textstyle\sum_{j=1}^{\kappa}}\big[\phi(u_j)+\rho(1\!-u_j)\|U_{j}\|
   +\phi(v_j)+\rho(1\!-v_j)\|V_{j}\|\big]\ge r.
  \end{align}
  Fix an arbitrary $\rho>\overline{\rho}$
  and an arbitrary $(U,V,u,v)\in\mathbb{R}^{m\times\kappa}
  \times\mathbb{R}^{n\times \kappa}\times [0,e]\times[0,e]$. Write
  \[
    J:=\big\{j\ |\ \phi(u_j)+\rho(1\!-u_j)\|U_j\|+\phi(v_j)+\rho(1\!-v_j)\|V_j\|\ge 2\big\}
    \ {\rm and}\ \overline{J}=\{1,\ldots,\kappa\}\backslash J.
  \]
  Clearly, if $|J|\ge r$ or $\frac{\mu}{4}\|U^{\mathbb{T}}U-V^{\mathbb{T}}V\|_F^2\ge r$,
  the stated conclusion automatically holds. We next consider
  the case that $|J|<r$ and $\frac{\mu}{4}\|U^{\mathbb{T}}U-V^{\mathbb{T}}V\|_F^2<r$.
  Notice that $(\overline{U},\overline{V},e,e)$ for any
  $(\overline{U},\overline{V})\in\mathcal{W}^*$ is a feasible solution of
  \eqref{MPEC-balance} with the objective value equal to $\kappa$,
  while the optimal value of \eqref{MPEC-balance} is $r$.
  Hence, $r\leq\kappa$, and we have $\overline{J}\ne\emptyset$.
  For each $j\in\overline{J}$,
  \begin{equation}\label{temp-ineq33}
    \|U_{\!j}V_{j}^{\mathbb{T}}\|_F=\|U_{j}\|\|V_{j}\|\leq \frac{\phi'_-(1)}{\rho}\sqrt{\Big(\frac{\phi'_-(1)}{\rho}\Big)^2+\frac{2\sqrt{r}}{\sqrt{\mu}}}.
  \end{equation}
  Indeed, for each $j\in\overline{J}$, it holds that
  $\phi(u_j)+\rho(1-u_j)\|U_{j}\|<1$ or $\phi(v_j)+\rho(1-v_j)\|V_{j}\|<1$.
  Together with $\phi(u_j)-1=\phi(u_j)-\phi(1)\ge\phi_{-}'(1)(u_j-1)$ and
  $\phi(v_j)-1\ge\phi_{-}'(1)(v_j-1)$, it follows that
  $\rho\|U_{j}\|<\phi_-'(1)$ or $\rho\|V_{j}\|<\phi_-'(1)$.
  Notice that $|\|U_{j}\|^2-\|V_{j}\|^2|<\frac{2\sqrt{r}}{\sqrt{\mu}}$
  since $\frac{\mu}{4}\|U^{\mathbb{T}}U-V^{\mathbb{T}}V\|_F^2<r$.
  Then, the inequality \eqref{temp-ineq33} follows. Thus, we have
 \begin{align*}
  \big\|\mathcal{A}(UV^{\mathbb{T}}-M)\big\|
  &\ge \big\|\mathcal {A}(U_{J}V_{J}^{\mathbb{T}}-M)\big\|
       -\big\|\mathcal {A}(U_{\overline{J}}V_{\overline{J}}^{\mathbb{T}})\big\|\nonumber\\
  &\ge\sqrt{\alpha}\big\|U_{J}V_{J}^{\mathbb{T}}-M\big\|_F
       -\|\mathcal {A}\|\big\|U_{\overline{J}}V_{\overline{J}}^{\mathbb{T}}\big\|_F\nonumber\\
  &\ge\sqrt{\alpha(r-|J|)}\sigma_r(M)-\sqrt{\kappa}\|\mathcal{A}\|\max_{1\le j\le \kappa}
  \big\|U_{\!j}V_{j}^{\mathbb{T}}\big\|_F\nonumber\\
  &\ge\sqrt{\alpha(r-|J|)}\sigma_r(M)-\frac{\|\mathcal {A}\|\sqrt{\kappa}\phi'_-(1)}{\rho}
       \sqrt{\Big(\frac{\phi'_-(1)}{\rho}\Big)^2+\frac{2\sqrt{r}}{\sqrt{\mu}}}\nonumber\\
   &>\sqrt{\frac{2(r-|J|)}{\nu}}
  \end{align*}
  where the third inequality is using $\min_{{\rm rank}(X)\le|J|}\|X-M\|_F^2
  =\sum_{i=|J|+1}^r\sigma_i^2(M)$, and the last one is due to $\rho>\phi'_-(1)$
  and $\rho\big(\sqrt{\alpha}\sigma_r(M)-\sqrt{2\nu^{-1}}\big)
  >\|\mathcal {A}\|\sqrt{\kappa}\phi'_-(1)\sqrt{1+\frac{2\sqrt{r}}{\sqrt{\mu}}}$.
  The last inequality implies that \eqref{UPC} holds.
  Recall that the optimal value of \eqref{MPEC-balance} equals $r$.
  By \cite[Definition 2.2]{LiuBP18}, the inequality \eqref{UPC} implies that
  the MPEC \eqref{MPEC-balance} is uniformly partial calm over its global optimal solution set,
  which by \cite[Proposition 2.1(a)]{LiuBP18} is equivalent to saying that
  \eqref{MPEC-penalty} is a global exact penalty of \eqref{MPEC-balance}.
 \end{proof}

 From the proof of Proposition \ref{Epenalty-LQ}, we see that the balanced
 term $\frac{\mu}{4}\|U^{\mathbb{T}}U-V^{\mathbb{T}}V\|_F^2$ in the function $F$
 plays a crucial role. When replacing $F$ with $\widetilde{f}$, it is unclear
 whether the result of Proposition \ref{Epenalty-LQ} holds or not, and we leave
 this topic for future study.
 \section{KL property of exponent $1/2$ of $\Psi$ and $\Theta_{\rho}$}

  In this section, for the function $f$ specified as in \eqref{least-squares}
  with $b=\mathcal{A}(M)$ for a matrix $M$ of rank $r$, we shall establish
  the KL property of exponent $1/2$ for the functions $\Psi$ and $\Theta_{\rho}$
  over the set of their global minimizers. This often requires the following inequality
  \begin{equation}\label{mat-ineq}
   \sigma_{\kappa}(A)\|B\|_F\le \|AB^{\mathbb{T}}\|_F\le\|A\|\|B\|_F
   \quad\ \forall A,B\in\mathbb{R}^{m\times\kappa}.
  \end{equation}
  For convenience, in the subsequent analysis we write $\sigma_i=\sigma_i(M)$ for $i=1,2,\ldots,m$.
  \subsection{KL property of exponent $1/2$ of $\Psi$}\label{sec4.1}

  To establish the KL property of exponent $1/2$ of $\Psi$ over the set of
  its global minimizer set, we first characterize the subdifferential of
  $\Psi$ at any point $(U,V)\in\mathbb{R}^{m\times\kappa}\times\mathbb{R}^{n\times\kappa}$.
 \begin{lemma}\label{subdiff-Psi}
  Fix an arbitrary $(U,V)\in\mathbb{R}^{m\times\kappa}\times\mathbb{R}^{n\times\kappa}$
  and write $J_U:=\{j\ |\ U_{j}\ne 0\}$ and  $J_V:=\{j\ |\ V_{j}\ne 0\}$.
  Then, $\widehat{\partial}\Psi(U,V)=\partial\Psi(U,V)=\partial_U\Psi(U,V)\times\partial_V\Psi(U,V)$ with
  \begin{subequations}
  \begin{align*}
  \partial_U\Psi(U,V)&=\!\Big\{G\in\mathbb{R}^{m\times \kappa}|\ G_{j}\!=\!\nu[\mathcal {A}^*\mathcal {A}(UV^{\mathbb{T}}\!-\!M)]V_{j}+\!\mu U(U^{\mathbb{T}}U_{j}\!-V^{\mathbb{T}}V_j),\,j\in J_U\Big\},\\
  \partial_V\Psi(U,V)&=\!\Big\{H\in \mathbb{R}^{n\times\kappa}|\ H_{j}\!=\!\nu[\mathcal{A}^*\mathcal {A}(UV^{\mathbb{T}}\!-\!M)]^{\mathbb{T}}U_{j}+\!\mu V(V^{\mathbb{T}}V_{j}\!-\!U^{\mathbb{T}}U_{j}),\, j\in J_V\Big\}.
 \end{align*}
 \end{subequations}
 \end{lemma}
 \begin{proof}
  By the expression of $F$ in \eqref{FUV}, $F$ is continuously
  differentiable in $\mathbb{R}^{m\times\kappa}\times\mathbb{R}^{n\times\kappa}$.
  Let $g(Z):=\|Z\|_{2,0}$ for $Z\in\mathbb{R}^{m\times\kappa}$.
  From \cite[Exercise 8.8(c)\& Proposition 10.5]{RW98},
  \begin{align*}
   \widehat{\partial}\Psi(U,V)=\widehat{\partial}_U\Psi(U,V)\times\widehat{\partial}_V\Psi(U,V)
   =\big(\nabla_UF(U,V)+\widehat{\partial}g(U)\big)\times
    \big(\nabla_VF(U,V)+\widehat{\partial}g(V)\big),\\
   \partial\Psi(U,V)=\partial_U\Psi(U,V)\times\partial_V\Psi(U,V)
   =\big(\nabla_UF(U,V)+\partial g(U)\big)\times
    \big(\nabla_VF(U,V)+\partial g(V)\big).
  \end{align*}
  By invoking \cite[Proposition 10.5]{RW98} and Lemma \ref{subdiff-L2norm1},
  it immediately follows that
  \[
    \widehat{\partial}g(U)
    =\partial g(U)=S_1\times\cdots\times S_\kappa
    \ \ {\rm with}\ S_j=\left\{\begin{array}{cl}
    \{0\}^n & {\rm if}\ j\in J_U;\\
    \mathbb{R}^n & {\rm if}\ j\notin J_U.
    \end{array}\right.
  \]
  Similarly, $\widehat{\partial}g(V)=\partial g(V)$ has
  such a characterization. Thus, we get the result.
 \end{proof}

  The following lemma provides a kind of stability for the global minimizer set
  of $\Psi$, i.e., for each global minimizer $(\overline{U},\overline{V})$
  of $\Psi$, there exists a neighborhood such that every point pair in this neighborhood
  has the same nonzero columns.
 \begin{lemma}\label{lemma41}
  Suppose that $\mathcal{X}^*\!\cap\Omega_{\kappa}\ne\emptyset$ and
  the $2r$-restricted smallest eigenvalue $\alpha$ of $\mathcal{A}^*\mathcal{A}$
  satisfies $\alpha>\frac{2}{\nu\sigma_r^2}$.
  Fix an arbitrary $(\overline{U},\overline{V})\in\mathcal{W}^*$. Then, for any
  $(U,V)\in\mathbb{R}^{m\times\kappa}\times\mathbb{R}^{n\times\kappa}$ with
  $\|[U-\overline{U}\ \ V-\overline{V}]\|_F<\frac{\sqrt{\sigma_r}}{4}$ and
  $0<\Psi(U,V)-\Psi(\overline{U},\overline{V})<\frac{1}{2}$, it holds that
  \[
    \big\{j\ |\ \overline{U}_j\ne0\big\}=\big\{j\ |\ \overline{V}_j\ne 0\big\}
    =\big\{j\ |\ U_j\ne 0\big\}=\big\{j\ |\ V_j\ne 0\big\}.
  \]
 \end{lemma}
 \begin{proof}
  Write $J\!:=\big\{j\ |\ \overline{U}_j\ne0\big\},
  J_U\!:=\{j\ |\ U_j\ne 0\}$ and $J_V\!:=\{j\ |\ V_j\ne 0\}$.
  By using Proposition \ref{MW-set}, we have $J=\big\{j\ |\ \overline{V}_j\ne 0\big\}$,
  $\|\overline{U}_j\|=\|\overline{V}_j\|$ for each $j\in J$, and $|J|=r$.
  Let $\{e_1,\ldots,e_{\kappa}\}$ be the orthonormal basis of $\mathbb{R}^{\kappa}$.
  Then, for each $j\in J$, we have
  \[
   \|\overline{U}_{\!j}\|^2=\|\overline{U}e_j\|^2=\langle e_je_j^{\mathbb{T}}, \overline{U}_J^{\mathbb{T}} \overline{U}_J\rangle
    \geq \sigma_r(\overline{U}_J^{\mathbb{T}} \overline{U}_J) =\sigma_r(\overline{U}^{\mathbb{T}} \overline{U}).
  \]
  This implies that
  \(
    \min_{j\in J}\|\overline{U}_{\!j}\|=\min_{j\in J}\|\overline{U}e_j\|\ge\sigma_r(\overline{U})
    =\sqrt{\sigma_r},
  \)
  where the last equality is using $\sigma(\overline{U})=\sigma(\overline{V})$ and
  $\sigma(\overline{U}\overline{V}^{\mathbb{T}})=\sigma(M)$.
  Together with $\|\overline{U}_{\!j}\|=\|\overline{V}_{\!j}\|$ for each $j\in J$,
  \[
    \min_{j\in J}\|\overline{U}_{\!j}\|=\min_{j\in J}\|\overline{V}_{\!j}\|
    \ge\sqrt{\sigma_r}.
  \]
  For each $j\in J$, since $\|U_j\|=\|U_j-\overline{U}_j+\overline{U}_j\|
  \ge\|\overline{U}_j\|-\!\|U-\overline{U}\|_F\ge
  \sqrt{\sigma_r}-\frac{\sqrt{\sigma_r}}{4}>\frac{1}{2}\sqrt{\sigma_r}$,
  we have $J\subseteq J_{U}$. Similarly, $J\subseteq J_{V}$.
  In addition, from $0<\Psi(U,V)-\Psi(\overline{U},\overline{V})<1/2$
  and $\Psi(\overline{U},\overline{V})=r$, it follows that
  $r<\Psi(U,V)<r+\frac{1}{2}$. By the expression of $\Psi$, we have
  \(
    r=|J|\le \frac{1}{2}(\|U\|_{2,0}+\|V\|_{2,0})<\Psi(U,V)<r+\frac{1}{2},
  \)
  which implies that $\|U\|_{2,0}+\|V\|_{2,0}=2r$. Together with
  $J\subseteq J_{U}$ and $J\subseteq J_{V}$, we conclude that
  $J= J_{U}=J_{V}$.
 \end{proof}
 \begin{theorem}\label{KLL20}
  Let $\alpha$ and $\beta$ be the $2r$-restricted smallest and largest eigenvalues
  of $\mathcal{A}^*\mathcal{A}$, respectively.
  Suppose that $\mathcal{X}^*\cap\Omega_{\kappa}\ne\emptyset$ and
  $\frac{\beta}{\alpha}<\frac{128\sigma_1^2(4\sigma_1+\sigma_r)^2+\sigma_r^4}
  {128\sigma_1^2(4\sigma_1+\sigma_r)^2-\sigma_r^4}$ with
  $ \alpha>\frac{4}{\nu\sigma_r^2}$. Fix an arbitrary
  $(\overline{U},\overline{V})\in\mathcal{W}^*$. Then,
  for any $(U,V)$ with $\|[U\!-\overline{U}\ \ V\!-\overline{V}]\|_F
  <\frac{\sqrt{\sigma_r}}{4}$ and $0<\Psi(U,V)-\Psi(\overline{U},\overline{V})<\frac{1}{2}$,
  the following inequality holds
  \begin{equation}\label{KL-Psi}
   {\rm dist}^2(0,\partial \Psi(U,V))\ge\gamma\big[\Psi(U,V)-\Psi(\overline{U},\overline{V})\big].
  \end{equation}
  with $\gamma=\min\big(\frac{\nu}{\beta}\big[\frac{(\beta+\alpha)\sigma_r^4}{128\sqrt{\sigma_1^3}(4\sigma_1+\sigma_r)^2}
       -\!\sqrt{\sigma_1}(\beta-\alpha)\big]^2,2\mu\sigma_r\big)$.
 \end{theorem}
 \begin{proof}
  Fix an arbitrary $(U,V)$ with $\|[U\!-\overline{U}\ \ V\!-\overline{V}]\|_F
  <\frac{\sqrt{\sigma_r}}{4}$ and $0<\Psi(U,V)-\Psi(\overline{U},\overline{V})<\frac{1}{2}$.
  Write $\Delta_U=U\!-\overline{U}$ and $\Delta_V=V\!-\overline{V}$.
  Then $\|[\Delta_{U}\ \Delta_{V}]\|_F<\frac{1}{4}\sqrt{\sigma_r}$.
  Let $J, J_{U}$ and $J_{V}$ be the index sets defined as in the proof of Lemma \ref{lemma41}.
  By Lemma \ref{subdiff-Psi},
 \begin{align}\label{dist-ineq41}
  {\rm dist}^2(0,\partial \Psi(U,V))
  &={\rm dist}^2(0,\partial_U \Psi(U,V))+{\rm dist}^2(0,\partial_V \Psi(U,V))\nonumber\\
  &=\nu^2\big[\|[\mathcal {A}^*\mathcal{A}(UV^{\mathbb{T}}-M)]V\|_F^2
     +\|U^{\mathbb{T}}[\mathcal {A}^*\mathcal {A}(UV^{\mathbb{T}}-M)]\|_F^2\big]\nonumber\\
  &\quad +\mu^2\big[\|U(U^{\mathbb{T}}U-V^{\mathbb{T}}V)\|_F^2+\|(V^{\mathbb{T}}V-U^{\mathbb{T}}U)V^{\mathbb{T}}\|_F^2\big].
 \end{align}
 Since $J_{U}=J$ by Lemma \ref{lemma41}, we have
 $\|U(U^{\mathbb{T}}U-V^{\mathbb{T}}V)\|_F^2=\|U_{J}(U_J^{\mathbb{T}}U_{J}-V_{J}^{\mathbb{T}}V_{J})\|_F^2$.
 Together with \eqref{mat-ineq}, we have
 $\|U(U^{\mathbb{T}}U-V^{\mathbb{T}}V)\|_F^2\ge\sigma_r^2(U_J)\|U_J^{\mathbb{T}}U_{J}-V_{J}^{\mathbb{T}}V_{J}\|_F^2$.
 Notice that $\sigma_r(U_{J})=\sigma_r(\overline{U}_{J}+[\Delta_U]_J)
 \ge\sigma_r(\overline{U}_{\!J})-\sigma_1(\Delta_U)>\sqrt{\sigma_r}-\frac{\sqrt{\sigma_r}}{4}>\sqrt{\sigma_r/2}$.
 Hence,
 \[
   \|U(U^{\mathbb{T}}U-V^{\mathbb{T}}V)\|_F^2\ge\frac{\sigma_r}{2}\|U_J^{\mathbb{T}}U_{J}-V_{J}^{\mathbb{T}}V_{J}\|_F^2
   =\frac{\sigma_r}{2}\|U^{\mathbb{T}}U-V^{\mathbb{T}}V\|_F^2.
 \]
 Following the same arguments, we also have
 $\|(V^{\mathbb{T}}V-U^{\mathbb{T}}U)V^{\mathbb{T}}\|_F^2
 \ge \frac{\sigma_r}{2}\|U^{\mathbb{T}}U-V^{\mathbb{T}}V\|_F^2$.
 By combining with these inequalities and the inequality \eqref{dist-ineq41},
 it follows that
 \begin{equation}\label{dist-ineq42}
  {\rm dist}^2(0,\partial \Psi(U,V))
  \ge \mu^2\sigma_r\|U^{\mathbb{T}}U-V^{\mathbb{T}}V\|_F^2.
 \end{equation}
  Next we proceed the arguments by two cases $UV^{\mathbb{T}}=M$ and
  $UV^{\mathbb{T}}\ne M$, respectively.

  \medskip
  \noindent
  {\bf Case 1:} $UV^{\mathbb{T}}=M$. In this case,
  $\Psi(U,V)-\Psi(\overline{U},\overline{V})=\frac{\mu}{4}\|U^{\mathbb{T}}U-V^{\mathbb{T}}V\|_F^2$.
  Together with the inequality \eqref{dist-ineq42},
  the desired inequality holds with $\gamma=4\mu\sigma_r$.

  \medskip
  \noindent
  {\bf Case 2:} $UV^{\mathbb{T}}\!\ne M$. Since $\|U\|\le\!\|\overline{U}\|+\!\|\Delta_U\|
  \le\!\frac{5}{4}\sqrt{\sigma_1}\!\le\!\sqrt{2\sigma_1}$ and $\|V\|\le\!\sqrt{2\sigma_1}$, we have
  \begin{align*}
   &\quad 2\sqrt{\sigma_1}\sqrt{\|[\mathcal {A}^*\mathcal{A}(UV^{\mathbb{T}}-M)]V\|_F^2
     +\|U^{\mathbb{T}}[\mathcal {A}^*\mathcal {A}(UV^{\mathbb{T}}-M)]\|_F^2}\\
   &\ge \sqrt{2\sigma_1}\big[\|[\mathcal {A}^*\mathcal{A}(UV^{\mathbb{T}}-M)]V\|_F
   +\|U^{\mathbb{T}}[\mathcal {A}^*\mathcal {A}(UV^{\mathbb{T}}-M)]\|_F\big]\\
   &\ge\big[\|[\mathcal {A}^*\mathcal{A}(UV^{\mathbb{T}}-M)]VV^{\mathbb{T}}\|_F
   +\|UU^{\mathbb{T}}[\mathcal {A}^*\mathcal {A}(UV^{\mathbb{T}}-M)]\|_F\big]\\
   &\ge\frac{1}{\|UV^{\mathbb{T}}-M\|_F}\langle[\mathcal{A}^*\mathcal{A}(UV^{\mathbb{T}}-M)]VV^{\mathbb{T}}, UV^{\mathbb{T}}-M\rangle\nonumber\\
   &\quad +\frac{1}{\|UV^{\mathbb{T}}-M\|_F}\langle UU^{\mathbb{T}}[\mathcal{A}^*\mathcal{A}(UV^{\mathbb{T}}-M)], UV^{\mathbb{T}}-M\rangle\nonumber\\
   &=\frac{1}{\|UV^{\mathbb{T}}-M\|_F}\langle\mathcal{A}(UV^{\mathbb{T}}-M),
      \mathcal{A}[(UV^{\mathbb{T}}-M)VV^{\mathbb{T}}]\rangle\nonumber\\
   &\quad +\frac{1}{\|UV^{\mathbb{T}}-M\|_F}\langle \mathcal {A}(UV^{\mathbb{T}}-M),
   \mathcal{A}[UU^{\mathbb{T}}(UV^{\mathbb{T}}-M)]\rangle.
  \end{align*}
  Notice that ${\rm rank}([UV^{\mathbb{T}}-M\ \ (UV^{\mathbb{T}}-M)VV^{\mathbb{T}}])\le 2r$.
  Applying Lemma \ref{RIPXY} to the two terms on the right hand side, we can obtain that
  \begin{align*}
   &\quad 2\sqrt{\sigma_1}\sqrt{\|[\mathcal {A}^*\mathcal{A}(UV^{\mathbb{T}}-M)]V\|_F^2
     +\|U^{\mathbb{T}}[\mathcal {A}^*\mathcal {A}(UV^{\mathbb{T}}-M)]\|_F^2}\\
   &\ge \frac{\beta+\alpha}{2\|UV^{\mathbb{T}}\!-M\|_F}\langle UV^{\mathbb{T}}\!-M,  (UV^{\mathbb{T}}\!-M)VV^{\mathbb{T}}\rangle-\frac{\beta-\alpha}{2}\|(UV^{\mathbb{T}}\!-M)VV^{\mathbb{T}}\|_F\nonumber\\
   &\quad + \frac{\beta+\alpha}{2\|UV^{\mathbb{T}}\!-M\|_F}\langle UV^{\mathbb{T}}\!-M,  UU^{\mathbb{T}}(UV^{\mathbb{T}}\!-M)\rangle-\frac{\beta-\alpha}{2}\|UU^{\mathbb{T}}(UV^{\mathbb{T}}\!-M)\|_F\nonumber\\
   &\ge\frac{(\beta+\alpha)[\|(UV^{\mathbb{T}}\!-M)V\|_F^2+\|U^{\mathbb{T}}(UV^{\mathbb{T}}\!-M)\|_F^2]}
       {2\|UV^{\mathbb{T}}\!-M\|_F}-2\sigma_1(\beta-\alpha)\|UV^{\mathbb{T}}\!-M\|_F
  \end{align*}
  where the last inequality is due to \eqref{mat-ineq}, $\|U\|\le\sqrt{2\sigma_1}$
  and $\|V\|\le\sqrt{2\sigma_1}$. Along with \eqref{dist-ineq41},
  \begin{align}\label{temp-ineq42}
   &\quad \frac{4\sqrt{\sigma_1}}{\nu(\beta+\alpha)}{\rm dist}(0,\partial \Psi(U,V))\nonumber\\
   &\ge \frac{\|(UV^{\mathbb{T}}\!-\!M)V\|_F^2+\|U^{\mathbb{T}}(UV^{\mathbb{T}}\!-\!M)\|_F^2}{\|UV^{\mathbb{T}}-M\|_F}
      -\frac{4\sigma_1(\beta-\alpha)}{\beta+\alpha}\|UV^{\mathbb{T}}\!-\!M\|_F.
  \end{align}
  To further deal with the right hand side of \eqref{temp-ineq42},
  we take $(P,Q)\in\mathbb{O}^{m,n}(M)$ and write
  \[
    \Sigma_{+}={\rm diag}(\sigma_1,\ldots,\sigma_r),\
    \widehat{U}=P^{\mathbb{T}}U,\,\widehat{V}=Q^{\mathbb{T}}V,\
    \widehat{U}_{\!J}=\left(\begin{matrix}
       \widehat{U}_{I\!J}\\ \widehat{U}_{\overline{I}\!J}
       \end{matrix}\right)\ {\rm and}\
    \widehat{V}_{J}=\left(\begin{matrix}
       \widehat{V}_{I\!J}\\ \widehat{V}_{\overline{I}\!J}
       \end{matrix}\right)
  \]
  where $\widehat{U}_{I\!J}\in\mathbb{R}^{r\times r}$ and
  $\widehat{U}_{\overline{I}\!J}\in\mathbb{R}^{(m-r)\times r}$
  are the matrix consisting of the first $r$ rows and
  the last $m-r$ rows of $\widehat{U}_{\!J}$, respectively,
  and similar is for $\widehat{V}_{I\!J}\in\mathbb{R}^{r\times r}$ and
  $\widehat{V}_{\overline{I}\!J}\in\mathbb{R}^{(m-r)\times r}$.
  Then, by using \eqref{mat-ineq}, $\|\widehat{U}\|=\|U\|\le\sqrt{2\sigma_1}$
  and $\|\widehat{V}\|=\|V\|\le\sqrt{2\sigma_1}$, we obtain
  \begin{align}\label{fl201}
  \|UV^{\mathbb{T}}-M\|_F
   &=\|\widehat{U}\widehat{V}^{\mathbb{T}}-{\rm Diag}(\sigma(M))\|_F
   =\left\|\left(\begin{matrix}
          \widehat{U}_{I\!J}\widehat{V}_{I\!J}^{\mathbb{T}}-\Sigma_+
            & \widehat{U}_{I\!J}\widehat{V}_{\overline{I}\!J}^{\mathbb{T}} \\
              \widehat{U}_{\overline{I}\!J}\widehat{V}_{I\!J}^{\mathbb{T}}
            & \widehat{U}_{\overline{I}\!J}\widehat{V}_{\overline{I}\!J}^{\mathbb{T}} \\
            \end{matrix}\right)\right\|_F\nonumber\\
  &\le \|\widehat{U}_{I\!J}\widehat{V}_{I\!J}^{\mathbb{T}}-\Sigma_+\|_F
       +\sqrt{2\sigma_1}\|\widehat{U}_{\overline{I}\!J}\|_F+\sqrt{2\sigma_1}\|\widehat{V}_{\overline{I}\!J}\|_F.
 \end{align}
 Notice that $\sigma_r(\widehat{V}_{J})=\sigma_r(V_{J})>\sqrt{\frac{\sigma_r}{2}}$
 and ${\rm rank}(\widehat{V}_{J})={\rm rank}(V_{J})=r$.
 Hence, it holds that
 \begin{align}\label{dfl201}
  \|(UV^{\mathbb{T}}\!-\!M)V\|_F
  &=\|(\widehat{U}\widehat{V}^{\mathbb{T}}\!-\!{\rm Diag}(\sigma(M)))\widehat{V}\|_F
  =\left\|\left(\begin{matrix}
   \!(\widehat{U}_{I\!J}\widehat{V}_{I\!J}^{\mathbb{T}}-\Sigma_+)\widehat{V}_{I\!J}
   +\widehat{U}_{I\!J}\widehat{V}_{\overline{I}\!J}^{\mathbb{T}}\widehat{V}_{\overline{I}\!J} \\
   \widehat{U}_{\overline{I}\!J}\widehat{V}_{J}^{\mathbb{T}}\widehat{V}_{J} \\
   \end{matrix}\right)\right\|_F\nonumber\\
 &\ge \frac{1}{\sqrt{2}}\big[\|(\widehat{U}_{I\!J}\widehat{V}_{I\!J}^{\mathbb{T}}-\Sigma_+)\widehat{V}_{I\!J}
   +\widehat{U}_{I\!J}\widehat{V}_{\overline{I}\!J}^{\mathbb{T}}\widehat{V}_{\overline{I}\!J}\|_F
   +\|\widehat{U}_{\overline{I}\!J}\widehat{V}_{J}^{\mathbb{T}}\widehat{V}_{J}\|_F\big]\nonumber\\
 &\ge \frac{\sigma_r}{2\sqrt{2}}\|\widehat{U}_{\overline{I}\!J}\|_F
     +\frac{1}{\sqrt{2}}\big\|(\widehat{U}_{I\!J}\widehat{V}_{I\!J}^{\mathbb{T}}-\Sigma_+)\widehat{V}_{I\!J}
   +\widehat{U}_{I\!J}\widehat{V}_{\overline{I}\!J}^{\mathbb{T}}\widehat{V}_{\overline{I}\!J}\big\|_F.
 \end{align}
 Since $\|\widehat{U}_{I\!J}\widehat{V}_{I\!J}^{\mathbb{T}}-\Sigma_+\|_F
 \le\|\widehat{U}\widehat{V}^{\mathbb{T}}-{\rm Diag}(\sigma(M))\|_F=\|UV^{\mathbb{T}}\!-M\|_F
 \le\frac{1}{\sqrt{\alpha}}\|\mathcal {A}(UV^{\mathbb{T}}-M)\|\le\frac{1}{\sqrt{\alpha\nu}}$
 where the last inequality is due to $\Psi(U,V)-\Psi(\overline{U},\overline{V})<1/2$,
 using $\alpha\ge\frac{4}{\nu\sigma_r^2}$ then yields
 $\|\widehat{U}_{I\!J}\widehat{V}_{I\!J}^{\mathbb{T}}\!-\Sigma_+\|_F\le\frac{\sigma_r}{2}$,
 which in turn implies that $\sigma_r(\widehat{U}_{I\!J}\widehat{V}_{I\!J}^{\mathbb{T}})\ge\frac{\sigma_r}{2}$.
 Consequently,
 \begin{align}\label{dfl202}
  \|(UV^{\mathbb{T}}\!-M)V\|_F
  &\ge\frac{\sigma_r}{2\sqrt{2}}\|\widehat{U}_{\overline{I}\!J}\|_F
     +\frac{1}{\sqrt{2}}\big\|(\widehat{U}_{I\!J}\widehat{V}_{I\!J}^{\mathbb{T}}-\Sigma_+)\widehat{V}_{I\!J}\big\|_F
     -\frac{1}{\sqrt{2}}\big\|\widehat{U}_{I\!J}\widehat{V}_{\overline{I}\!J}^{\mathbb{T}}\widehat{V}_{\overline{I}\!J}\big\|_F\nonumber\\
 &\ge \frac{\sigma_r}{2\sqrt{2}}\|\widehat{U}_{\overline{I}\!J}\|_F
      +\frac{1}{\sqrt{2}\|\widehat{U}_{I\!J}\|}\big\|(\widehat{U}_{I\!J}\widehat{V}_{I\!J}^{\mathbb{T}}-\Sigma_+)
      \widehat{V}_{I\!J}\widehat{U}_{I\!J}^{\mathbb{T}}\|_F
      -\!\frac{\|\widehat{U}_{I\!J}\widehat{V}_{\overline{I}\!J}^{\mathbb{T}}\widehat{V}_{\overline{I}\!J}\|_F}{\sqrt{2}}\nonumber\\
 &\ge \frac{\sigma_r}{2\sqrt{2}}\|\widehat{U}_{\overline{I}\!J}\|_F +\frac{\sigma_r}{2\sqrt{2\sigma_1}}\|\widehat{U}_{I\!J}\widehat{V}_{I\!J}^{\mathbb{T}}-\Sigma_+\|_F
 -\sqrt{2}\sigma_1\|\widehat{V}_{\overline{I}\!J}^{\mathbb{T}}\|_F
 \end{align}
 where the last inequality is also using $\|\widehat{U}_{I\!J}\|\le\sqrt{2\sigma_1}$
 and $\|\widehat{V}_{\overline{I}\!J}\|\le\sqrt{2\sigma_1}$. Similarly,
 \begin{subequations}
  \begin{align}\label{dfl203}
   \|U^{\mathbb{T}}(UV^{\mathbb{T}}\!-M)\|_F
   &\ge \frac{\sigma_r}{2\sqrt{2}}\|\widehat{V}_{\overline{I}\!J}\|_F
        +\frac{1}{\sqrt{2}}\big\|\widehat{U}_{I\!J}^{\mathbb{T}}(\widehat{U}_{I\!J}\widehat{V}_{I\!J}^{\mathbb{T}}-\Sigma_+)
         +\widehat{U}_{\overline{I}\!J}^{\mathbb{T}}\widehat{U}_{\overline{I}\!J}\widehat{V}_{I\!J}^{\mathbb{T}}\big\|_F,\\
   \label{dfl204}
  \|U^{\mathbb{T}}(UV^{\mathbb{T}}\!-M)\|_F
  &\ge\frac{\sigma_r}{2\sqrt{2}}\|\widehat{V}_{\overline{I}\!J}\|_F
      +\frac{\sigma_r}{2\sqrt{2\sigma_1}}\|\widehat{U}_{I\!J}\widehat{V}_{I\!J}^{\mathbb{T}}-\Sigma_+\|_F
      -\sqrt{2}\sigma_1\|\widehat{U}_{\overline{I}\!J}\|_F.
  \end{align}
 \end{subequations}
  From \eqref{dfl202}, we have
  $\|\widehat{U}_{\overline{I}\!J}\|_F \le\frac{2\sqrt{2}}{\sigma_r}\|(UV^{\mathbb{T}}\!-M)V\|_F$.
  Together with \eqref{dfl204}, it follows that
 \begin{align*}
  &\Big(\frac{4\sigma_1}{\sigma_r}+1\Big)\big[\|U^{\mathbb{T}}(UV^{\mathbb{T}}\!-M)\|_F + \|(UV^{\mathbb{T}}\!-M)V\|_F\big]\nonumber\\
  &\ge\frac{\sigma_r}{2\sqrt{2}}\|\widehat{V}_{\overline{I}\!J}\|_F
      +\frac{\sigma_r}{2\sqrt{2\sigma_1}}\|\widehat{U}_{I\!J}\widehat{V}_{I\!J}^{\mathbb{T}}-\Sigma_+\|_F
      +\frac{\sigma_r}{2\sqrt{2}}\|\widehat{U}_{\overline{I}\!J}\|_F\nonumber\\
  &\ge\frac{\sigma_r}{4\sqrt{\sigma_1}}\|UV^{\mathbb{T}}-M\|_F,
 \end{align*}
 where the last inequality is due to \eqref{fl201}. From this, we immediately obtain that
 \begin{align}\label{UVM}
  &\|U^{\mathbb{T}}(UV^{\mathbb{T}}-M)\|_F^2+\|(UV^{\mathbb{T}}-M)V\|_F^2\nonumber\\
  &\ge\frac{1}{2}\big[\|U^{\mathbb{T}}(UV^{\mathbb{T}}-M)\|_F+ \|(UV^{\mathbb{T}}-M)V\|_F\big]^2\nonumber\\
  &\ge\frac{\sigma_r^4}{32\sigma_1(4\sigma_1+\sigma_r)^2}\|UV^{\mathbb{T}}-M\|_F^2.
 \end{align}
 Combining the last inequality and \eqref{temp-ineq42} and
 using $\frac{1}{\sqrt{\beta}}\|\mathcal{A}(UV^{\mathbb{T}}\!-M)\|\!\leq\!\|UV^{\mathbb{T}}\!-\!M\|_F$
 gives
 \begin{align*}
  \frac{4\sqrt{\sigma_1}}{\nu(\beta+\alpha)}{\rm dist}(0,\partial \Psi(U,V))
  &\ge\Big[\frac{\sigma_r^4}{32\sigma_1(4\sigma_1+\sigma_r)^2}-\frac{4\sigma_1(\beta-\alpha)}{\beta+\alpha}\Big]\|UV^{\mathbb{T}}-M\|_F\nonumber\\
  &\ge\Big[\frac{\sigma_r^4}{32\sigma_1(4\sigma_1+\sigma_r)^2}-\frac{4\sigma_1(\beta-\alpha)}{\beta+\alpha}\Big]
       \frac{1}{\sqrt{\beta}}\|\mathcal{A}(UV^{\mathbb{T}}\!-M)\|
 \end{align*}
 where the last inequality is using $\frac{\beta}{\alpha}<\frac{128\sigma_1^2(4\sigma_1+\sigma_r)^2+\sigma_r^4}
 {128\sigma_1^2(4\sigma_1+\sigma_r)^2-\sigma_r^4}$. Together with \eqref{dist-ineq42},
 we have
 \begin{align*}
  {\rm dist}^2(0,\partial \Psi(U,V))
  &\ge\frac{\nu^2}{2\beta}\Big[\frac{(\beta+\alpha)\sigma_r^4}{128\sigma_1\sqrt{\sigma_1}(4\sigma_1+\sigma_r)^2}
       -\sqrt{\sigma_1}(\beta-\alpha)\Big]^2
   \|\mathcal{A}(UV^{\mathbb{T}}\!-M)\|^2\\
  &\quad +\frac{\mu^2\sigma_r}{2}\|U^{\mathbb{T}}U-V^{\mathbb{T}}V\|_F^2\\
  &\ge\gamma\Big[\frac{\nu}{2}\|\mathcal{A}(UV^{\mathbb{T}}\!-M)\|^2
  +\frac{\mu}{4}\|U^{\mathbb{T}}U-V^{\mathbb{T}}V\|_F^2\Big].
 \end{align*}
 Notice that $\Psi(U,V)-\Psi(\overline{U},\overline{V})=\frac{\nu}{2}\|\mathcal{A}(UV^{\mathbb{T}}\!-M)\|^2
  +\frac{\mu}{4}\|U^{\mathbb{T}}U-V^{\mathbb{T}}V\|_F^2$. Together with the result of Case 1,
  we obtain the desired conclusion.
 \end{proof}
 \begin{remark}
  Theorem \ref{KLL20} establishes the KL property of exponent $1/2$ of
  the function $\Psi$ over the set of global minimizers under a suitable
  assumption on the $2r$-restricted condition number of $\mathcal{A}^*\mathcal{A}$.
  Along with Proposition \ref{MW-set}, under the assumptions of Theorem \ref{KLL20}
  one may seek the unique global optimal solution $M$ of rank not more than $\kappa$
  in a linear rate when solving the problem \eqref{prob-balance} with a starting
  point not far from the set $\mathcal{W}^*$.
%
 \end{remark}

  Inspired by the result of Theorem \ref{KLL20}, it is natural to
  ask whether the function $\Psi$ has the KL property of exponent $1/2$
  in the set of critical points. The following example shows that $\Psi$
  does not have the KL property of exponent $1/2$ at those critical points
  for which the number of nonzero columns is greater than the rank of $M$.
  \begin{example}\label{example41}
   Take $\mathcal{A}(X)={\rm vec}(X)$ for $X\in\mathbb{R}^{4\times 4}$,
   where ${\rm vec}(X)$ represents the vector obtained by arranging $X$
   in terms of its columns, $M=4E$ with $E\in\mathbb{R}^{4\times 4}$
   being a matrix of all ones, and $\kappa=4$. Now the problem \eqref{prob-balance}
   is specified as follows
   \begin{align*}
    \min_{U,V\in\mathbb{R}^{4\times 4}}
    \left\{\Psi(U,V):=\frac{\nu}{2}\big\|UV^{\mathbb{T}}\!-M\big\|_F^2
    +\frac{\mu}{4}\|U^{\mathbb{T}}U\!-V^{\mathbb{T}}V\|_F^2
     +\frac{1}{2}(\|U\|_{2,0}+\|V\|_{2,0})\right\}.
   \end{align*}
   Consider $\overline{U}=\overline{V}=E$. It is easy to check that
   $(\overline{U}, \overline{V})$ is a regular critical point of $\Psi$.
   For any $t\in(0,1)$, define $U(t)=V(t):=\overline{U}+t\Delta$ with
   $\Delta\in\mathbb{R}^{4\times 4}$ defined as follows
   \[
     \Delta=\left[\begin{matrix}
                    -E+2I & 0_{2\times 2}\\
                   0_{2\times 2}& 0_{2\times 2}
                 \end{matrix}\right].
   \]
   We calculate that  $\Psi(U(t),V(t))-\Psi(\overline{U},\overline{V})=8\nu t^4$
   and ${\rm dist}(0,\partial\Psi(U(t),V(t)))=8\sqrt{2}\nu t^3$.
   This shows that $\Psi$ can not have the KL property of exponent $1/2$
   at those critical points with the number of nonzero columns greater than $r$.
   Notice that $(\overline{U},\overline{V})$ is also a global minimizer of
   the function $(U,V)\mapsto\frac{\nu}{2}\big\|UV^{\mathbb{T}}\!-M\big\|_F^2
   +\frac{\mu}{4}\|U^{\mathbb{T}}U\!-V^{\mathbb{T}}V\|_F^2$.
   Then, the above arguments imply that this function may not have
   the KL property of exponent $1/2$ over the set of its global minimizers.
  \end{example}
 \subsection{KL property of exponent $1/2$ of $\Theta_{\rho}$}\label{sec4.2}

  In this part, we establish the KL property of exponent $1/2$ for the function
  $\Theta_{\rho}$ over its global minimizer set, where $\Theta_{\rho}$ is
  defined by \eqref{prob-DC} with $\theta$ satisfying Assumption \ref{assump} below.
  Such an assumption is mild and the functions in \cite[Example 3-5]{LiuBP18}
  all satisfy it.
 \begin{assumption}\label{assump}
  The $\theta(t):=|t|-\psi^*(|t|)$ with $\phi\in\!\mathscr{L}$ is concave
  and nondecreasing on $[0,+\infty)$, and there exist constants
  $\varpi>0$ and $c>0$ such that $\theta'(t)\ge c$ for all $t\in(0,\varpi)$.
 \end{assumption}

 To achieve the KL property of exponent of such $\Theta_{\rho}$,
 we need the following lemma. Since its proof is similar to that of
 Lemma \ref{subdiff-Psi} by Lemma \ref{subdiff-L2norm2}, we here omit it.
 \begin{lemma}\label{subdiff-Thetarho}
  Let $\Theta_{\rho}$ be the function in equation \eqref{prob-DC} associated to $\rho>0$.
  Fix an arbitrary $(U,V)\in\mathbb{R}^{m\times\kappa}\times\mathbb{R}^{n\times\kappa}$.
  Write $J_U\!:=\{j\ |\ U_j\ne 0\},J_V\!:=\{j\ |\ V_j\ne 0\},\overline{J}_U=\{1,\ldots,\kappa\}\backslash J_U$
  and $\overline{J}_V=\{1,\ldots,\kappa\}\backslash J_V$.
  Then, $\partial\Theta_{\rho}(U,V)\!=\partial_U\Theta_{\rho}(U,V)\times\partial_V\Theta_{\rho}(U,V)$ with
  \begin{subequations}
  \begin{align*}\label{subdiffU-Thetarho}
  \partial_U\Theta_{\rho}(U,V)&\subseteq\!\bigg\{G\in\mathbb{R}^{m\times \kappa}|\ G_j\!=\!\nu[\mathcal {A}^*\mathcal {A}(UV^{\mathbb{T}}\!-\!M)]V_j+\!\mu U(U^{\mathbb{T}}U_j\!-V^{\mathbb{T}}V_j)+S_j,\\
    &\qquad\left. S_{j}=\frac{\theta'(\|U_{j}\|)U_j}{2\|U_{j}\|}\ {\rm for}\ j\in J_U,
            S_{j}\in \frac{\theta'(\|U_{j}\|)}{2}D^*g(U_{j})\ {\rm for}\ j\in\overline{J}_U\right\},\\
  \partial_V\Theta_{\rho}(U,V)&\subseteq\bigg\{H\in \mathbb{R}^{n\times\kappa}|\ H_{j}\!=\!\nu[\mathcal{A}^*\mathcal {A}(UV^{\mathbb{T}}\!-\!M)]^{\mathbb{T}}U_{j}
   +\!\mu V(V^{\mathbb{T}}V_{j}\!-\!U^{\mathbb{T}}U_{j})+T_{j},\\
   &\qquad\left. T_{j}=\frac{\theta'(\|V_{j}\|)V_{j}}{2\|V_{j}\|}\ {\rm for}\ j\in J_V,
     T_{j}\in \frac{\theta'(\|V_{j}\|)}{2} D^*g(V_{j})\ {\rm for}\ j\in\overline{J}_V\right\}
 \end{align*}
 \end{subequations}
 where $g(z)=\|z\|$ for $z\in\mathbb{R}^{\kappa}$, and $D^*g(U_{j})$
 is the coderivative of $g$ at $U_{j}$.
 \end{lemma}

 \begin{theorem}\label{KL-Thetarho}
  Let $\Theta_{\rho}$ be the function in \eqref{prob-DC} associated to
  $\rho>\overline{\rho}$ with $\theta$ satisfying Assumption \ref{assump},
  and let $\alpha$ and $\beta$ be the $2r$-restricted smallest and largest eigenvalues
  of $\mathcal{A}^*\mathcal{A}$, respectively. Suppose that
  $\mathcal{X}^*\!\cap\Omega_{\kappa}\ne\emptyset$ and
  $\frac{\beta}{\alpha}<\frac{128\sigma_1^2(4\sigma_1+\sigma_r)^2+\sigma_r^4}
  {128\sigma_1^2(4\sigma_1+\sigma_r)^2-\sigma_r^4}$ with $\alpha\!>\frac{4}{\nu\sigma_r^2}$.
  Fix an arbitrary $(\overline{U},\overline{V})\in \mathcal{W}^*$.
  Then, for any $(U,V)$ with $0<\Theta_{\rho}(U,V)-\Theta_{\rho}(\overline{U},\overline{V})<\frac{1}{2}$
  and $\|[U\!-\overline{U}\ \ V\!-\overline{V}]\|_F<\min(\frac{\sqrt{\sigma_r}}{4},
  \frac{\varpi}{\rho},\frac{c\rho}{4\sqrt{\nu}\|\mathcal {A}\|+16\mu\sigma_1})$,
  the following inequality
  \[
    {\rm dist}^2(0,\partial\Theta_{\rho}(U,V))
    \ge\gamma'(\Theta_{\rho}(U,V)-\Theta_{\rho}(\overline{U},\overline{V}))
  \]
 holds with $\gamma'=\min\big(\frac{\nu}{\beta}\big[\frac{(\beta+\alpha)\sigma_r^4}{128\sqrt{\sigma_1^3}(4\sigma_1+\sigma_r)^2}
       -\!\sqrt{\sigma_1}(\beta-\alpha)\big]^2,2\mu\sigma_r,\frac{32}{c^2\rho^2}\big)$.
 \end{theorem}
 \begin{proof}
  Fix an arbitrary $(U,V)$ with
  $\|[U-\overline{U}\ \ V-\overline{V}]\|_F<\min(\frac{\sqrt{\sigma_r}}{4},
  \frac{\varpi}{\rho},\frac{c\rho}{4\sqrt{\nu}\|\mathcal {A}\|+16\mu\sigma_1})$
  and $0<\Theta_{\rho}(U,V)-\Theta_{\rho}(\overline{U},\overline{V})<\frac{1}{2}$.
  Write $J\!:=\big\{j\ |\ \overline{U}_j\ne0\big\},\Delta_U=U-\overline{U}$
  and $\Delta_V=V-\overline{V}$.
  By Proposition \ref{Epenalty-LQ}, it follows that $\mathcal{W}^*$ is exactly
  the global minimizer set of $\Theta_{\rho}$. Using the same arguments as those
  for Lemma \ref{lemma41}, we have $\|U_{j}\|\ge \frac{1}{2}\sqrt{\sigma_r}$
  and $\|V_{j}\|\ge \frac{1}{2}\sqrt{\sigma_r}$ for each $j\in J$,
  and $J\subseteq J_{U}$ and $J\subseteq J_{V}$. Moreover, from
  $\|[U-\overline{U}\ \ V-\overline{V}]\|_F<\frac{\sqrt{\sigma_r}}{4}$,
  it follows that $\|U\|\le \sqrt{2\sigma_1}$ and $\|V\|\le \sqrt{2\sigma_1}$.
  In addition, for each $j\in J$ it holds that
  \begin{equation}\label{varphij}
    \theta(\rho\|U_{j}\|)=\theta(\rho\|V_{j}\|)=1\ \ {\rm and}\ \
    \theta'(\rho\|U_{j}\|)=\theta'(\rho\|V_{j}\|)=0.
  \end{equation}
  Indeed, by the definition of $\overline{\rho}$ in Proposition \ref{Epenalty-LQ},
  it is immediate to have
  \[
    \rho>\overline{\rho}\ge\frac{4\|\mathcal {A}\|\sqrt{\kappa}}{\sqrt{\alpha}\sigma_r-\sqrt{2\nu^{-1}}}\phi_-'(1)
    \ge\frac{4\|\mathcal {A}\|\sqrt{\kappa}\phi_-'(1)}{\sqrt{\alpha}\sigma_r-\sqrt{2\nu^{-1}}}
    \cdot\frac{\phi_-'(1)}{\rho}
  \]
  where the last inequality is by $\rho>\phi_{-}'(1)$. So,
  $\rho^2\sigma_r\ge\rho^2\big(\sigma_r-\!\sqrt{\frac{2}{\nu\alpha}}\big)
  \ge\frac{4\|\mathcal {A}\|\sqrt{\kappa}[\phi_-'(1)]^2}{\sqrt{\alpha}}$.
  Since $\|A\|\ge\sqrt{\beta}\ge\sqrt{\alpha}$, we have $\frac{1}{2}\rho\sqrt{\sigma_r}>\phi_{-}'(1)$.
  Together with $\|U_{j}\|\ge \frac{1}{2}\sqrt{\sigma_r}$ for each $j\in J$,
  we have $\rho\|U_{j}\|>\phi_{-}'(1)$ for each $j\in J$.
  Similarly, $\rho\|V_{j}\|>\phi_{-}'(1)$ for each $j\in J$.
  By the definition of $\theta$ and $\theta^*$, we calculate that
  $\theta(\rho\|U_{j}\|)=\theta(\rho\|V_{j}\|)=\psi(1)=1$
  for each $j\in J$. In fact, by noting that $\theta$
  is locally Lipschitz everywhere since $\psi^*$ is locally Lipschitz
  everywhere by ${\rm dom}\psi^*=\mathbb{R}$ and its convexity,
  there exists a neighborhood of $\rho\|U_{\cdot j}\|$ in which
  $\theta(t)\equiv 1$, which implies that
  $\theta'(\rho\|U_{j}\|)=0$ for each $j\in J$. Similarly,
  $\theta'(\rho\|V_{j}\|)=0$ for each $j\in J$. Thus,
  the equalities in \eqref{varphij} hold.
  Note that $0<\Theta_{\rho}(U,V)-\Theta_{\rho}(\overline{U},\overline{V})<\frac{1}{2}$
  and $\Theta_{\rho}(\overline{U},\overline{V})=r$.
  By the expression of $\Theta_{\rho}(U,V)$ and the first equality
  in \eqref{varphij}, we have $\frac{\nu}{2}\|\mathcal {A}(UV^{\mathbb{T}}-M)\|^2\le\frac{1}{2}$.
  Now we proceed the arguments by three cases.

  \medskip
  \noindent
  {\bf Case 1:} $J\ne J_{U}$. Now $U_{j}=[\Delta_U]_{j}$
  and $V_{j}=[\Delta_V]_{j}$ for each $j\in J_U\backslash J$.
  By using  Assumption \ref{assump}, $\|\mathcal{A}(UV^{\mathbb{T}}\!-M)\|\le1/\sqrt{\nu}$
  and $\max(\|U\|,\|V\|)\le\sqrt{2\sigma_1}$, for each $j\in J_U\backslash J$,
  \begin{align*}
   &\big\|\nu[\mathcal {A}^*\mathcal{A}(UV^{\mathbb{T}}-M)]V_{j}
   +\mu U(U^{\mathbb{T}}U_{j}-V^{\mathbb{T}}V_{j})
   +\frac{\rho U_{j}}{2\|U_{j}\|}\theta'(\rho\|U_{j}\|)\big\|\\
   &\ge \frac{1}{2}c\rho -\sqrt{\nu}\|\mathcal {A}^*\|\|V_{j}\|
        -2\mu\sigma_1(\|U_{j}\|+\|V_{j}\|)\\
   &\ge \frac{1}{2}c\rho-(\sqrt{\nu}\|\mathcal {A}\|+4\mu\sigma_1)
        \|[\Delta_U\ \ \Delta_V]\|_F
   \ge\frac{c\rho}{2}-\frac{c\rho}{4}=\frac{c\rho}{4}
  \end{align*}
  where the second one is using $\|[\Delta_U\ \ \Delta_V]\|_F
  \le\frac{c\rho}{4\sqrt{\nu}\|\mathcal {A}\|+16\mu\sigma_1}$.
  Along with Lemma \ref{subdiff-Thetarho},
  \begin{align*}
   {\rm dist}^2(0,\partial_U\Theta_{\rho}(U,V))
  &\ge\sum_{j\in J_U\backslash J}\big[\frac{c\rho}{2}-\!\sqrt{\nu}\|\mathcal{A}\|\|V_{j}\|
      -2\mu\sigma_1(\|U_{j}\|+\|V_{j}\|)\big]^2\nonumber\\
  &\ge\Big[\frac{c\rho}{2}-(\sqrt{\nu}\|\mathcal {A}\|+4\mu\sigma_1)\|[\Delta_U\ \ \Delta_V]\|_F\Big]^2
  \ge\frac{c^2\rho^2}{16},
 \end{align*}
 which together with $\Theta_{\rho}(U,V)-\Theta_{\rho}(\overline{U},\overline{V})<1/2$ implies that
 \begin{equation}\label{Case1-result}
  {\rm dist}^2(0,\partial\Theta_{\rho}(U,V))
   \geq{\rm dist}^2(0,\partial_U \Theta_{\rho}(U,V))
   \ge\frac{32}{c^2\rho^2}\big[\Theta_{\rho}(U,V)-\Theta_{\rho}(\overline{U},\overline{V})\big].
  \end{equation}

  \noindent
  {\bf Case 2:} $J\ne J_{V}$. Using the same arguments
  as those for Case 1 yields \eqref{Case1-result}.

  \medskip
  \noindent
  {\bf Case 3:} $J=J_U=J_{V}$. In this case, by invoking \eqref{varphij} and
  Lemma \ref{subdiff-Thetarho}, we have
  \begin{align*}
  {\rm dist}^2(0,\partial \Psi(U,V))
  &={\rm dist}^2(0,\partial_U \Psi(U,V))+{\rm dist}^2(0,\partial_V \Psi(U,V))\nonumber\\
  &=\nu^2\big[\|[\mathcal {A}^*\mathcal{A}(UV^{\mathbb{T}}-M)]V\|_F^2
     +\|U^{\mathbb{T}}[\mathcal {A}^*\mathcal {A}(UV^{\mathbb{T}}-M)]\|_F^2\big]\nonumber\\
  &\quad +\mu^2\big[\|U(U^{\mathbb{T}}U-V^{\mathbb{T}}V)\|_F^2+\|(V^{\mathbb{T}}V-U^{\mathbb{T}}U)V^{\mathbb{T}}\|_F^2\big].
 \end{align*}
  In addition, $\Theta_{\rho}(U,V)-\Theta_{\rho}(\overline{U},\overline{V})
  =\frac{\nu}{2}\|\mathcal{A}(UV^{\mathbb{T}}-M)\|^2+\frac{\mu}{4}\|U^{\mathbb{T}}U-V^{\mathbb{T}}V\|_F^2$.
  Hence, using the same arguments as those for Theorem \ref{KLL20}, we can obtain
  \[
    {\rm dist}^2(0,\partial\Theta_{\rho}(U,V))
    \ge\gamma(\Theta_{\rho}(U,V)-\Theta_{\rho}(\overline{U},\overline{V})).
  \]
  Together with Case 1 and Case 2, the desired result follows.
 \end{proof}

 Finally, it is worthwhile to point out that when the linear operator $\mathcal{A}$
 is obtained by full sampling, i.e., $\mathcal{A}^*\mathcal{A}(X)=X$ for all
 $X\in\mathbb{R}^{m\times n}$, the functions $\Psi$ and $\Theta_{\rho}$
 associated to $\nu>\frac{4}{\sigma_r^2}$ have the KL property of exponent
 $1/2$ over the set of their global minimizers.
  \section{Numerical experiments}\label{sec5}

  In this section, we illustrate the KL property of exponent $1/2$ of
  $\Psi$ and $\Theta_{\rho}$ by applying the proximal linearized alternating
  minimization (PLAM) method to the problems \eqref{prob-balance} and \eqref{prob-DC},
  respectively, where $f$ is specified as in \eqref{least-squares} with $b=\mathcal{A}(M)$
  for a low-rank $M$. For the function $\Theta_{\rho}$, we choose
  $\phi(t):=\frac{a-1}{a+1}t^2+\frac{2}{a+1}t\ (a>1)$ for $t\in\mathbb{R}$.
  An elementary calculation yields that the conjugate function of $\psi$
  takes the following form
  \begin{equation}\label{psi-star}
  \psi^*(s)=\left\{\begin{array}{cl}
                      0 & \textrm{if}\ s\leq \frac{2}{a+1};\\
                      \frac{((a+1)s-2)^2}{4(a^2-1)} & \textrm{if}\ \frac{2}{a+1}<s\leq \frac{2a}{a+1};\\
                      s-1 & \textrm{if}\ s>\frac{2a}{a+1}.
                \end{array}\right.
 \end{equation}
  So, Assumption \ref{assump} holds with $\varpi=\frac{2}{a+1}$ and $c=1$.
  Set $\lambda=1/\nu$ and $\widetilde{\mu}=\mu/\nu$. Write
  $g_{\lambda,\rho}(t)\equiv \lambda\theta(\rho t)+\frac{\tau}{2}t^2$
  with $\tau=\frac{\lambda(a+1)\rho^2}{2(a-1)}$. Along with $\theta(t)=|t|-\psi^*(|t|)$,
  it is easy to check that $g_{\lambda,\rho}$ is convex.
  For the problem \eqref{prob-balance}, let $\Phi(U,V)\equiv f(U,V)+\frac{\widetilde{\mu}}{4}\|U^{\mathbb{T}}U\!-\!V^{\mathbb{T}}V\|_F^2$
  and $h(t)\equiv \lambda{\rm sign}(|t|)$; while for \eqref{prob-DC},
  let $\Phi(U,V)\!\equiv f(U,V)+\frac{\widetilde{\mu}}{4}\|U^{\mathbb{T}}U\!-V^{\mathbb{T}}V\|_F^2
  -\frac{\tau}{4}\big[\|U\|_F^2+\|V\|_F^2\big]$ and $h(t)\equiv g_{\lambda,\rho}(t)$.
  Then, the problems \eqref{prob-balance} and \eqref{prob-DC}
  can be written as
  \begin{align}\label{Eprob-DC}
   &\min_{U\in\mathbb{R}^{m\times \kappa},V\in\mathbb{R}^{n\times \kappa}}
   \Big\{\Phi(U,V)+\frac{1}{2}\sum_{j=1}^{\kappa}\big[h(\|U_j\|)+h(\|V_j\|)\big]\Big\}.
 \end{align}
  We apply the PLAM method \cite{Bolte14,Xu13} to solving \eqref{Eprob-DC}
  and its iterate steps are as follows.
 \begin{algorithm}[!h]
  \caption{\label{APLAM}{(\bf Accelerated PLAM method for solving \eqref{Eprob-DC})}}
  \textbf{Initialization:} Select an integer $\kappa\ge 1$, an appropriate $\lambda>0$,
  and constants $L_U,L_V>0$. Choose $(U^{-1},V^{-1})=(U^0,V^0)\in\mathbb{R}^{m\times\kappa}
  \times\mathbb{R}^{n\times\kappa}$ and $t_0=t_{-1}=1$. Set $k=0$.\\
 \textbf{while} the stopping conditions are not satisfied \textbf{do}
  \begin{itemize}
   \item Set $\widetilde{U}^{k}=U^{k}+\frac{t_{k-1}-1}{t_k}(U^{k}\!-\!U^{k-1})$ and
         $\widetilde{V}^{k}=V^{k}+\frac{t_{k-1}-1}{t_k}(V^{k}\!-\!V^{k-1})$;

   \item Solve the following two minimization problems
         \begin{subequations}
         \begin{align*}
          U^{k+1}\in\mathop{\arg\min}_{U\in\mathbb{R}^{m\times\kappa}}
           \Big\{\langle\nabla_{U}\Phi(\widetilde{U}^k,{V}^k),U\!-\!\widetilde{U}^k\rangle
           +\frac{L_U}{2}\|U\!-\!\widetilde{U}^k\|_F^2+\frac{\lambda}{2}\sum_{j=1}^{\kappa}h(\|U_j\|)\Big\},\\
           V^{k+1}\in\mathop{\arg\min}_{V\in\mathbb{R}^{n\times\kappa}}
           \Big\{\langle\nabla_{V}\Phi(U^{k+1},\widetilde{V}^k),V\!-\!\widetilde{V}^k\rangle
           +\frac{L_V}{2}\|V\!-\!\widetilde{V}^k\|_F^2+\frac{\lambda}{2}\sum_{j=1}^{\kappa}h(\|V_j\|)\Big\}.
           \end{align*}
         \end{subequations}

   \item Set $t_{k+1}:=\frac{1+\sqrt{1+4t_k^2}}{2}$ and $k\leftarrow k+1$.
  \end{itemize}
 \textbf{end while}
 \end{algorithm}
  \begin{remark}\label{remark-Alg}
  {\bf(i)} The constants $L_U$ and $L_V$ in Algorithm \ref{APLAM} are an upper
  estimation for the Lipschitz constant of $\nabla_U\Phi(\cdot,V)$ and
  $\nabla_V\Phi(U,\cdot)$, respectively, over a certain compact set of $(U,V)$
  which includes the iterate sequence $\{(U^k,V^k)\}$.

  \medskip
  \noindent
  {\bf(ii)} Whether $h(t)\equiv {\rm sign}(|t|)$ or $h(t)\equiv g_{\lambda,\rho}(t)$,
   one may easily achieve a global optimal solution of the subproblems since
   their proximal operators have a closed form. Notice that Algorithm \ref{APLAM}
   is an accelerated type of the PLAM proposed in \cite{Bolte14},
   and for its global convergence and linear rate of convergence analysis,
   the reader may refer to \cite{Xu13}.

  \medskip
  \noindent
  {\bf(iii)} By comparing the optimal conditions of the two subproblems with
  that of \eqref{Eprob-DC}, when
  \begin{subnumcases}{}\label{cond1}
    \frac{\|\nabla_{U}\Phi(\widetilde{U}^k,{V}^k)-\nabla_{U}\Phi(U^{k+1},{V}^{k+1})
    +L_U(U^{k+1}-\widetilde{U}^k)\|}{1+\|b\|}\le \epsilon,\\
    \label{cond2}
   \frac{\|\nabla_{V}\Phi(U^{k+1},\widetilde{V}^{k})-\nabla_{V}\Phi(U^{k+1},{V}^{k+1})
    +L_V(V^{k+1}-\widetilde{V}^k)\|}{1+\|b\|}\le \epsilon
  \end{subnumcases}
  holds for a pre-given tolerance $\epsilon>0$, we terminate Algorithm \ref{APLAM}
  at the iterate $(U^{k+1},V^{k+1})$.
  \end{remark}

  For the subsequent testing, the starting point $(U^0,V^0)$ of Algorithm \ref{APLAM}
  is always chosen as $(P{\rm Diag}(\sqrt{\sigma^{\kappa}(X^0)}),
  Q{\rm Diag}(\sqrt{\sigma^{\kappa}(X^0)}))$ with $(P,Q)\in\mathbb{O}^{m,n}(X^0)$
  for $X^0=\mathcal{A}^*(\mathcal {A}(M))$, where $\sigma^{\kappa}(X^0)\in\mathbb{R}^{\kappa}$
  is the vector consisting of the first $\kappa\ge r$ components of $\sigma(X^0)$.
  It should be emphasized that such a starting point is not close to the bi-factors
  of $M$ unless $\kappa=r$. Unless otherwise stated, the tolerance $\epsilon$ in
  \eqref{cond1}-\eqref{cond2} is chosen as $10^{-10}$. All numerical results are
  computed by a laptop computer running on 64-bit Windows Operating System
  with an Intel(R) Core(TM) i7-7700 CPU 2.8GHz and 16 GB RAM.
  \subsection{Illustration of the linear convergence}\label{subsec5.1}

  We take $M\in\mathbb{R}^{m\times n}$ with $m=n=4000$ and $r=10$
  for example to illustrate the linear convergence of the iterate sequence.
  For this purpose, we apply Algorithm \ref{APLAM} to the problem \eqref{Eprob-DC}
  associated to \eqref{prob-balance} with $\kappa=30r$, where the linear operator $\mathcal{A}$
  is obtained by the uniform sampling with the sample ratio $8.99\%$, and the parameters
  $\widetilde{\mu}$ and $\lambda$ are set as $10^{-3}$ and $150\|\mathcal{A}^*(\mathcal{A}(M))\|$,
  respectively. Figure \ref{fig1} plots the iteration error curves and the time curve,
  where $(U^f,V^f)$ is the final output of Algorithm \ref{APLAM}. Since
  the relative error $\|U^f(V^f)^{\mathbb{T}}-M\|_F/\|M\|_F\le 2.04\times 10^{-12}$,
  we conclude that $(U^f,V^f)$ is a global optimal solution of \eqref{prob-balance}.
  The subfigure on the left hand shows that the sequence $\{(U^k,V^k)\}$
  generated by Algorithm \ref{APLAM} indeed converges linearly to $(U^f,V^f)$.
  We also use Algorithm \ref{APLAM} to solve the problem \eqref{Eprob-DC}
  associated to \eqref{prob-DC} with $\kappa=30r$, where $\mathcal{A}$ is obtained
  by the uniform sampling with the sample ratio $7.49\%$, and the parameters
  $\widetilde{\mu},\lambda$ and $\rho$ are respectively set as
  $10^{-3},\frac{a+1}{2}(0.05\|\mathcal{A}^*(\mathcal{A}(M))\|)^2$
  and $\frac{40}{(a+1)\|\mathcal{A}^*(\mathcal{A}(M))\|}$ for $a=3.7$.
  Figure \ref{fig2} plots the iteration error curves and the time curve,
  and the iteration error curves show that the sequence $\{(U^k,V^k)\}$ is linearly convergent.
  Since the relative error $\|U^f(V^f)^{\mathbb{T}}-M\|_F/\|M\|_F\le 2.69\times 10^{-12}$,
  we conclude that the sequence $\{(U^k,V^k)\}$ also converges linearly to a global optimal
  solution $(U^f,V^f)$.

\begin{figure}[H]
  \setlength{\abovecaptionskip}{-0.2cm}
  \centerline{\epsfig{file=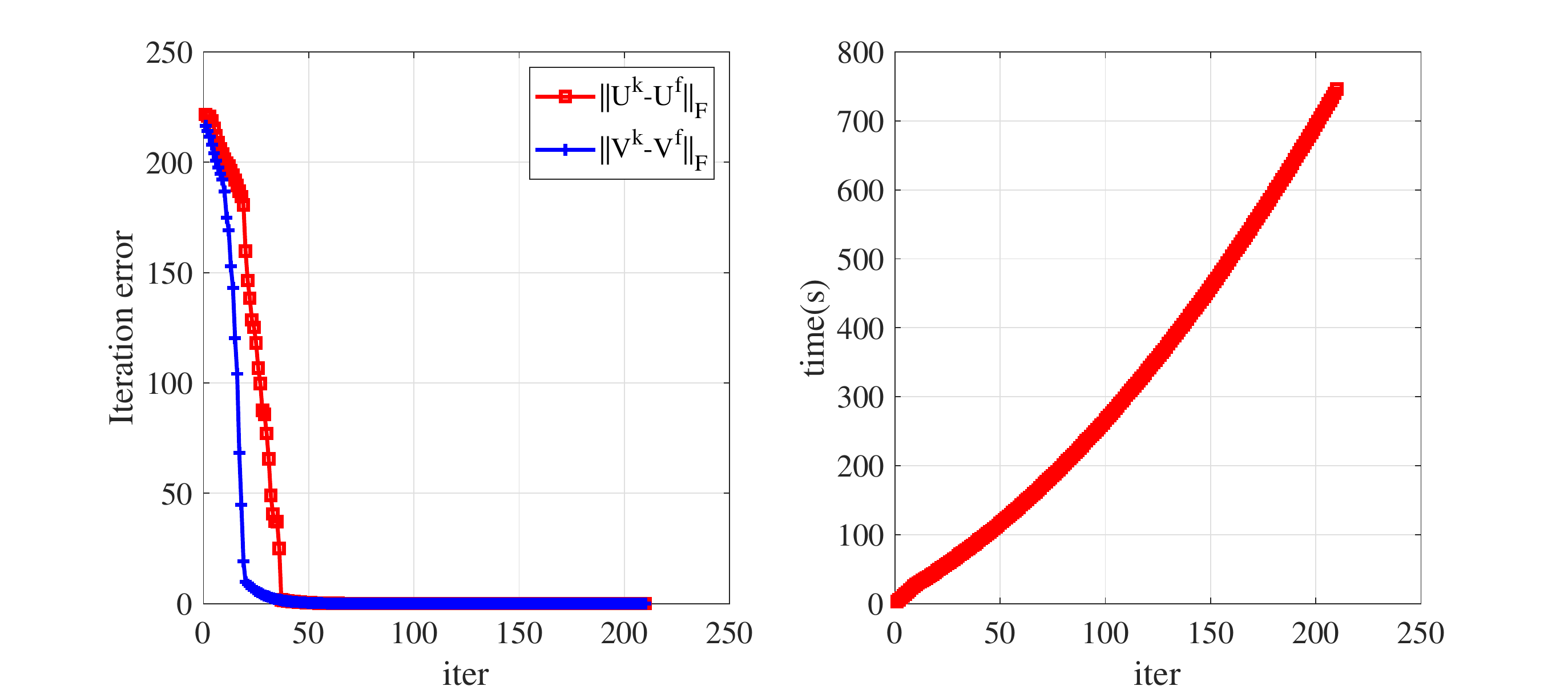,width=4.5in}}\par
 \caption{\small The iterate errors and computing time of Algorithm \ref{APLAM} for minimizing $\Psi$}
 \label{fig1}
 \end{figure}

 
 \begin{figure}[H]
  \setlength{\abovecaptionskip}{-0.2cm}
  \centerline{\epsfig{file=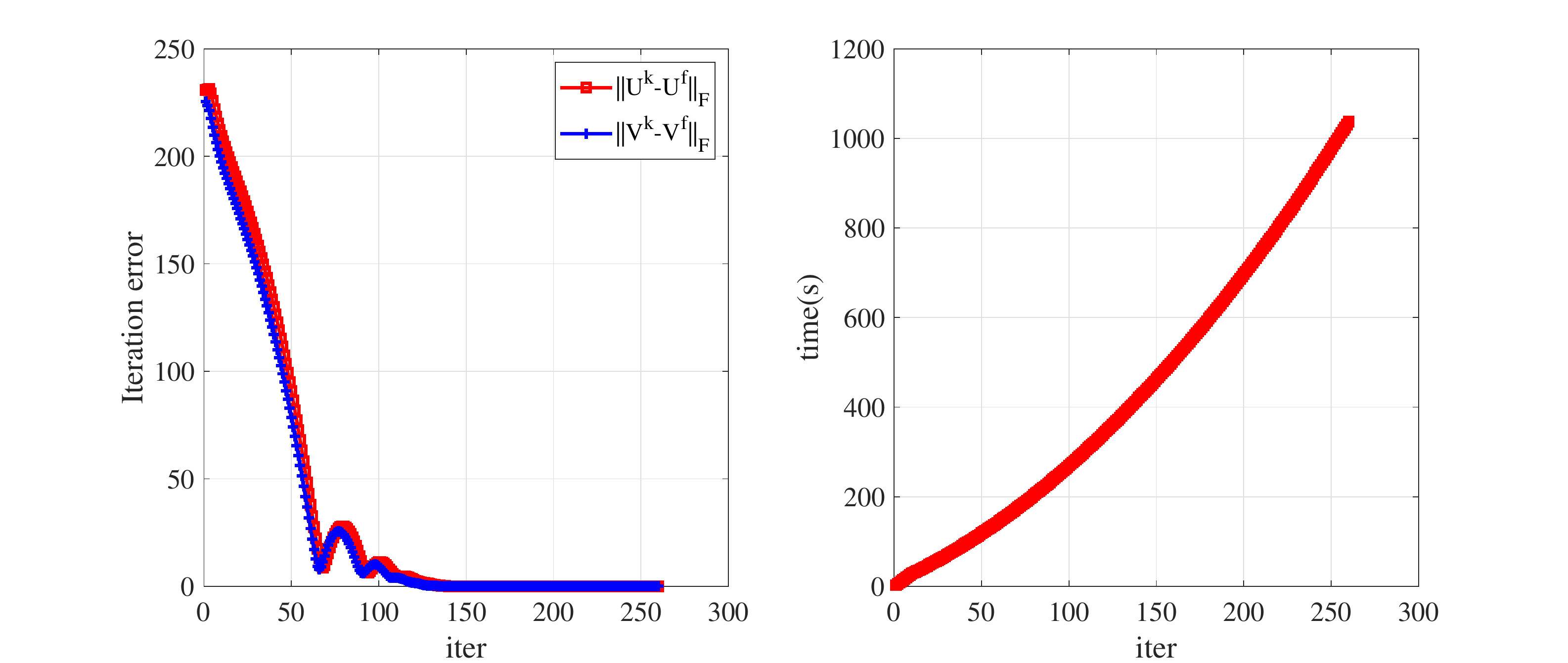,width=4.5in}}\par
 \caption{\small The iterate errors and computing time of Algorithm \ref{APLAM} for minimizing $\Theta_{\rho}$}
 \label{fig2}
 \end{figure}
 
 \subsection{Influence of $\lambda$ on the linear convergence}\label{subsec5.2}

  We take $M\in\mathbb{R}^{m\times n}$ with $m=n=2000$ and $r=10$ for example
  to illustrate the influence of $\lambda$ on the linear convergence of the iterates.
  To that end, we use Algorithm \ref{APLAM} to solve the problem \eqref{Eprob-DC}
  with $\kappa=20r$ and $\widetilde{\mu}=10^{-3}$, where the linear operator
  $\mathcal{A}$ is obtained by the uniform sampling with the sample ratio $13.3\%$.
  For the problem \eqref{Eprob-DC} associated to \eqref{prob-balance},
  we set $\lambda_i=c_i\|\mathcal{A}^*(\mathcal{A}(M))\|$,
  while for the problem \eqref{Eprob-DC} associated to \eqref{prob-DC} set
  $\lambda_i=\frac{a+1}{2}(c_i\|\mathcal{A}^*(\mathcal{A}(M))\|)^2$
  with $a=3.7$ and $\rho_i=\frac{2}{(a+1)c_i\|\mathcal{A}^*(\mathcal{A}(M))\|}$.
  Figure \ref{fig3} plots the iteration error curve $\|U^k-U^f\|_F$ corresponding to
  four different $\lambda_i\ (i=0,1,2,3)$.
   \begin{figure}[H]
  \setlength{\abovecaptionskip}{-0.2cm}
  \centerline{\epsfig{file=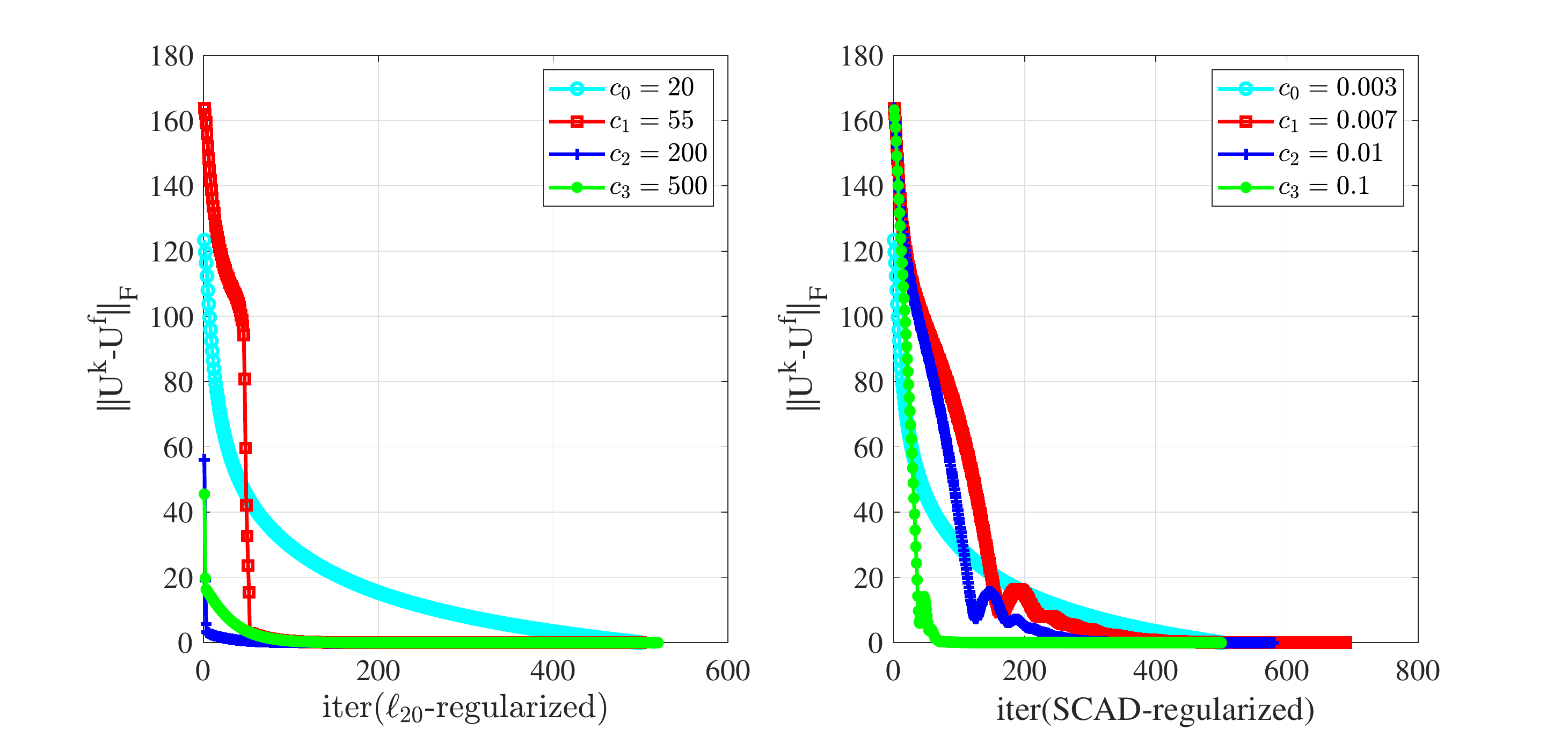,width=4.5in}}\par
 \caption{\small The iterate errors of Algorithm \ref{APLAM} for minimizing $\Psi$
  and $\Theta_{\rho}$ with different $\lambda$}
 \label{fig3}
 \end{figure}
%
  We see that, when minimizing $\Psi$ and $\Theta_{\rho}$ with $\lambda\ge\lambda_2$,
  the iterate sequence displays a linear convergence, but it does not have the linear
  convergence when minimizing $\Psi$ and $\Theta_{\rho}$ with $\lambda=\lambda_1$.
  We check that for $\lambda=\lambda_1$ the final output $(U^f,V^f)$ still has a full
  rank $200$ since $\lambda=\lambda_1$ is too small. This does not contradicts the result
  of Theorem \ref{KLL20} since our starting point is not close to the bi-factors of $M$.
  In addition, the iterate sequence also displays a linear convergence when
  minimizing $\Psi$ and $\Theta_{\rho}$ with $\lambda=\lambda_4$, though
  the final output does not recover $M$ since the relative error
  $\|U^f(V^f)^{\mathbb{T}}-M\|_F/\|M\|_F\ge 0.13$. Since the nonzero-columns of
  $U^f$ and $V^f$ are equal to $r$, this performance of this example does not
  contradicts the result of Example \ref{example41}. In other words,
  the functions $\Psi$ and $\Theta_{\rho}$ may have the KL property of exponent $1/2$
  over the set of stationary points for which the number of nonzero columns equals
  the rank of the true matrix $M$.
 \section{Conclusions}\label{sec6}

  For the rank regularized loss minimization model, we have proposed
  an $\ell_{2,0}$-norm regularized factored formulation and derived
  some equivalent DC surrogates from the global exact penalty of
  its MPEC reformulation. For these nonsmooth factored models,
  we take a local view to study the KL property of exponent $1/2$
  of their objective functions over the set of their global minimizers,
  which means that gradient descent or alternating minimization methods
  with a well-chosen starting point will converge linearly to an optimal solution.
  Numerical testing in Section \ref{sec5} indicates that the proposed regularized
  factored models are successful even if the chosen starting point is far from
  the set of global optima. This implies that a good geometric landscape may
  exist for these nonsmooth factored functions. We will leave this topic
  for our future study. In addition, we will focus on the KL property of
  exponent $1/2$ for $\Psi$ and $\Theta_{\rho}$ under the noisy setting.

\end{document}